\documentclass[reqno, 11pt]{amsart}
\usepackage{amssymb,amsmath,amsfonts}
\usepackage[margin=1in]{geometry}
\usepackage{dsfont}
\usepackage{color}
\usepackage{graphicx}
\usepackage{latexsym}
\usepackage{amsmath}
\usepackage{amssymb}
\usepackage{graphics}
\usepackage[dvips]{epsfig}
\usepackage{mathrsfs}
\usepackage{mathtools}
\usepackage{leqno}
\usepackage{stmaryrd,cite}

\usepackage{cases}
\usepackage{enumerate}
\usepackage[usenames,dvipsnames,svgnames]{xcolor}
\definecolor{refkey}{rgb}{1,0,0.5}
\definecolor{labelkey}{rgb}{0,0.4,1}
\usepackage{todonotes}
\usepackage{marginnote}

\newcommand{\Om}{\Omega}

\newcommand{\Div}{\operatorname{div}}
\newcommand{\av}[1]{\overline{#1}}
\newcommand{\norm}[2]{\lVert #1\rVert_{#2}}

\newcommand{\GN}{\mathcal N}

\presetkeys{todonotes}{fancyline, color=green!40}{}

\makeatletter
\renewcommand{\@todonotes@drawMarginNoteWithLine}{%
	\begin{tikzpicture}[remember picture, overlay, baseline=-0.75ex]%
		\node [coordinate] (inText) {};%
	\end{tikzpicture}%
	\marginnote[{
		\@todonotes@drawMarginNote%
		\@todonotes@drawLineToLeftMargin%
	}]{
		\@todonotes@drawMarginNote%
		\@todonotes@drawLineToRightMargin%
	}%
}
\makeatother
\usepackage{listings}
\usepackage{ifpdf}
\ifpdf
\usepackage[CJKbookmarks=true,
hyperindex=true,
pdfstartview=FitH,
bookmarksnumbered=true,
bookmarksopen=true,
colorlinks=true,
citecolor=blue,
linkcolor=blue,
urlcolor=blue,
pdfborder=001,
pdfauthor={},
pdftitle={},
pdfkeywords={},
]{hyperref}
\else
\usepackage[
hypertex,
hyperindex,
linkcolor=blue,
unicode,
citecolor=blue%
]{hyperref}
\fi
\usepackage{cleveref}

\allowdisplaybreaks

\numberwithin{equation}{section}

\newtheorem{thm}{Theorem}[section]

\newtheorem{lem}[thm]{Lemma}
\newtheorem{prop}[thm]{Proposition}
\newtheorem{rmk}[thm]{Remark}
\newtheorem{defn}[thm]{Definition}

\newcommand{\pl}{\partial}

\newcommand{\be}{\begin{equation}}
	\newcommand{\ee}{\end{equation}}

\newcommand{\bee}{\begin{equation*}}
	\newcommand{\eee}{\end{equation*}}

\newcommand{\bse}{\begin{subequations}}
	\newcommand{\ese}{\end{subequations}}

\newcommand{\bs}{\begin{split}}
	\newcommand{\es}{\end{split}}

\begin{document}
	\author{Yingzhi Du$^{1}$}\thanks{$^{1}$ School of Mathematics and Statistics, Wuhan University, No. 299, Bayi Road, Wuchang District, Wuhan City, Hubei Province, China, 430072. E-mail: yingzhidu2@gmail.com}
	
	\author{Huijiang Zhao$^{2}$}\thanks{$^{2}$ School of Mathematics and Statistics, Wuhan University, No. 299, Bayi Road, Wuchang District, Wuhan City, Hubei Province, China, 430072. E-mail: hhjjzhao@hotmail.com}

	\title[] {Global Existence of Strong and Weak Solutions to the 2D Compressible Navier-Stokes-Poisson System in Bounded Domains with Large Data and Vacuum}

	\begin{abstract}

			We study the two-dimensional compressible Navier--Stokes--Poisson system in a general simply connected bounded $C^{2,1}$ domain with Navier--slip boundary conditions and a positive, possibly non-flat, doping profile. We establish the global existence of a unique strong solution and of a global weak solution for arbitrarily large initial data allowing vacuum when the shear viscosity is a positive constant and the bulk one proportional to a power of the density with the power bigger than $\frac{4}{3}$.
			
			The main difficulties arise from the interaction between the Poisson force and the general boundary geometry. To overcome them, we introduce the modified material derivative
			$$
			\widetilde{u}:=\dot{u}-\nabla\Phi:= \pl_t u + (u \cdot \nabla)u - \nabla \Phi,
			$$
			and derive quantitative $L^p$ estimates for the density by means of a Neumann duality argument together with a harmonic correction generated by the Navier--slip boundary condition. We further obtain a pointwise representation of the effective viscous flux by pulling back the Neumann Green function on the unit disk through a conformal mapping. The resulting commutator estimates, combined with elliptic estimates for the Poisson equation, yield a uniform-in-time upper bound for the density and allow us to close the higher-order a priori estimates. Finally, the global weak solution is constructed through a charge-preserving approximation and compactness argument.\\

	
	\end{abstract}
	
	\maketitle

	\section{Introduction}

	The compressible Navier-Stokes-Poisson (NSP) system consists of the Navier-Stokes equations coupled with the self-consistent Poisson equation and is used to model the motion of charged particles (electrons or holes; see \cite{[23]} for further details). In two spatial dimensions, the NSP system for a single carrier type takes the form
	\begin{equation}\label{origin}
		\left\{
		\begin{array}{lll}
			\rho_{t}+\mbox{div}(\rho u)=0,\\[2mm]
			\rho\left(u_{t}+(u\cdot\nabla)u\right)+\nabla P=\mu\Delta u+\nabla((\mu+\lambda(\rho))\mbox{div}u)+\rho\nabla\Phi,\\[2mm]
			\Delta\Phi=\rho-b(x),
		\end{array}
		\right.
	\end{equation}
	where $\rho$, $u=(u_1(x,t),u_2(x,t))$ and $P$ denote the density, velocity field and pressure of the charged particles, respectively.
	The self-consistent electric potential $\Phi=\Phi(x,t)$ is coupled with the density through the Poisson equation.
	The pressure $P$ is given by
	$$P(\rho)=a_1\rho^{\gamma},$$
	where $a_1>0$ and $\gamma\geq1$ are constants. We impose the following assumptions on the shear viscosity coefficient $\mu$ and the bulk viscosity coefficient $\lambda(\rho)$:
	\begin{equation}
		0<\mu=\mbox{constant},\quad \lambda(\rho)=a_2 \rho^{\beta},\nonumber
	\end{equation}
	with positive constants $a_2$ and $\beta$. We set $a_1=a_2=1$ without loss of generality.
	The background profile $b=b(x)>0$ is prescribed and represents the sum of the background ion density and the net density of impurities.

	The aim of this paper is to investigate the global existence of solutions to the initial boundary value (IBV) problem for (\ref{origin}) in $(t,x)\in [0,+\infty)\times\Omega$, where $\Omega \subset \mathbb{R}^2$ is a simply connected bounded $C^{2,1}$-domain. The system is supplemented with the initial data
	\begin{equation}\label{initial}
		\rho(x,0)=\rho_0(x), \quad \rho u(x,0)=m_0(x), \quad \Delta \Phi(x,0)= \rho_0(x)-b(x) , \quad \int_{\Omega} \rho_0 \mbox{d}x = \int_{\Omega} b(x) \mbox{d}x, \quad x \in \Omega,
	\end{equation}
	and the Navier-slip boundary conditions
	\begin{equation}\label{boundary1}
		u\cdot n=0,\quad \mbox{curl}u = -A u  \cdot n^{\bot},\quad \mbox{on}\quad \partial\Omega,
	\end{equation}
	and
	\begin{equation}\label{boundary2}
		\nabla\Phi\cdot n\mid_{\partial\Omega}=0,
	\end{equation}
	where $A$ is a non-negative smooth function on the boundary and $n=(n_1,n_2)$ denotes the outward unit normal to $\partial\Omega$. Throughout the paper, all generic constants may depend on the fixed boundary coefficient $A$, on $n$, and on a fixed smooth extension of $(A,n,n^{\bot})$.
	The unit tangential vector $n^{\bot}$ on the boundary $\pl \Omega$ is defined by
	$$n^{\bot}:=(n_2,-n_1).$$
	We may smoothly extend $n$, $n^{\bot}$ and $A$ to the entire domain. These extensions are not unique, and we fix one choice throughout the paper.

	
	There is an extensive literature on the well-posedness theory for the
	multidimensional compressible Navier--Stokes equations. The local
	well-posedness theory for sufficiently regular solutions can be traced
	back to the classical works of Nash \cite{Nash1962} and Serrin
	\cite{Serrin1959}. The first global existence result for classical
	solutions near a non-vacuum equilibrium was established by
	Matsumura and Nishida \cite{MatsumuraNishida1980}. Hoff
	\cite{Hoff1995,Hoff2005} subsequently developed a theory allowing
	rougher initial data and introduced a number of important estimates
	based on the material derivative. For global weak solutions, a major
	breakthrough was achieved by Lions \cite{Lions1998}, whose restriction
	on the adiabatic exponent was later improved by Feireisl, Novotn\'y
	and Petzeltov\'a \cite{FeireislNovotnyPetzeltova2001}. For three-dimensional
	flows with vacuum and large oscillations, Huang, Li and Xin
	\cite{HuangLiXin2012} established global classical solutions under a
	smallness assumption on the initial energy.
	
	In contrast to the small-energy theory, global solvability for
	arbitrarily large initial data is much more delicate. A remarkable
	class of results is available in two space dimensions when the bulk
	viscosity depends on the density. More precisely, for
	\[
	\mu>0,\qquad \lambda(\rho)=\rho^\beta,
	\]
	Vaigant and Kazhikhov \cite{VK1995} first obtained the
	global existence of strong solutions for large initial data under the
	condition $\beta>3$, with the density initially bounded away from
	vacuum. Their result initiated the large-data theory for the
	two-dimensional density-dependent viscosity model. Related global
	weak solutions for the two-dimensional periodic problem were obtained
	by Perepelitsa \cite{P2006}, who also derived uniform upper and lower
	bounds for the density under the assumption that the initial density
	is bounded away from vacuum. Subsequently, Jiu, Wang and Xin
	\cite{JiuWangXin2014,JiuWangXin2018} extended the theory to allow
	vacuum states in the periodic problem and in the Cauchy problem. By developing refined estimates for the effective viscous
	flux and exploiting the commutator structure associated with the
	density equation, Huang and Li \cite{HL2016,HL2022} further reduced
	the restriction on the bulk viscosity exponent to
	\[
	\beta>\frac{4}{3}
	\]
	for the periodic problem and the Cauchy problem.
	
	The presence of a physical boundary introduces substantial additional
	difficulties, since the singular-integral structure available on
	$\mathbb R^2$ or $\mathbb T^2$ is no longer directly applicable.
	A fundamental advance in this direction was made by Fan, Li and Li
	\cite{FLL2022}, who considered the two-dimensional compressible
	Navier--Stokes equations in a general bounded simply connected domain
	subject to the Navier-slip boundary condition. They established the
	global existence of both strong and weak solutions with vacuum and
	arbitrarily large initial data for $\beta>4/3$. A key ingredient of
	their analysis is a pointwise representation of the effective viscous
	flux. By means of the Riemann mapping theorem and a pull-back of the
	Neumann Green function from the unit disc, the singular terms can be
	reduced to a commutator-type form, for which the boundary singularities
	can be controlled by suitable estimates of $\nabla u$. Their work
	provides the main analytical framework that motivates the present
	paper.
	
	
	We now turn to the compressible Navier--Stokes--Poisson system. The
	coupling with the Poisson equation introduces a nonlocal force and
	creates additional difficulties which do not occur in the usual
	compressible Navier--Stokes equations. Local solvability has been
	studied in several settings. For instance, Zhang and Tan
	\cite{ZhangTan2007} considered solutions to a two-dimensional
	compressible Navier--Stokes--Poisson system, while Tan and Zhang
	\cite{TanZhang2010} proved the local existence and uniqueness of
	strong solutions to the three-dimensional isentropic system, allowing
	the initial density to contain vacuum. Local strong solutions to the
	full Navier--Stokes--Poisson system with vacuum were also studied by
	Yin and Tan \cite{YinTan2009}.
	
	A considerable amount of work has also been devoted to global
	solutions. An early global weak-solution result for the Cauchy problem
	was obtained by Donatelli \cite{Donatelli2003}, where compactness
	properties induced by the Poisson equation play an important role.
	Further global weak and strong solutions, as well as their large-time
	behavior, have been investigated under various assumptions on the
	pressure, viscosity coefficients and doping profile; see, for example,
	\cite{DucometNecasovaVasseur2010,YeDou2017,LiuXuZhang2020,LiuXuZhangCorr2021}
	and the references therein. Of particular relevance here, Liu, Xu and Zhang
	\cite{LiuXuZhang2020,LiuXuZhangCorr2021} considered the three-dimensional Cauchy problem
	with a large and non-flat doping profile and established global strong
	solutions with vacuum and large oscillations, provided that the
	initial energy is sufficiently small.
	
	The initial-boundary value problem is even more sensitive to the
	interaction between the Poisson force and the boundary geometry.
	In two dimensions, Li, Zhou and Liao \cite{LiZhouLiao2018} studied
	the Navier--Stokes--Poisson system in a bounded domain in connection
	with the zero-electron-mass limit and obtained global regular solutions
	under well-prepared assumptions. In three-dimensional bounded domains, global solutions near a
	non-constant steady state were obtained by Liu and Zhong
	\cite{LiuZhong2021}. Liu, Si and Yu \cite{LiuSiYu2023} subsequently
	established global strong solutions with slip boundary conditions and
	vacuum under a small-initial-energy assumption. Chen, Huang and Shi \cite{ChenHuangShi2024} subsequently
	established global classical solutions in three-dimensional bounded
	domains with slip boundary conditions, allowing a large and non-flat
	doping profile, large oscillations and vacuum, again under a
	small-energy condition. Related global solutions near a non-constant
	steady state were obtained in \cite{Wu2025}.
	
	More recently, progress has also been made toward genuinely large-data
	problems. For the three-dimensional Cauchy problem with
	density-dependent viscosities, global strong solutions with
	arbitrarily large initial data were established under a sufficiently
	large doping profile in \cite{NSPLargeData2025}. Global classical
	solutions for arbitrarily large data have also been obtained on solid
	balls under spherical symmetry and a different density-dependent
	viscosity structure in \cite{FanHuangLei2026}. These results indicate
	that the large-data theory for the Navier--Stokes--Poisson system is
	strongly dependent on both the geometry and the structure of the
	viscosity coefficients.
	
	To the best of our knowledge, however, there has been no result on
	global strong and weak solutions with arbitrarily large initial data
	and vacuum for the two-dimensional compressible Navier--Stokes--Poisson
	system in a general bounded simply connected domain under the
	Navier-slip boundary condition, when
	\[
	\mu>0,\qquad
	\lambda(\rho)=\rho^\beta,\qquad
	\beta>\frac{4}{3},
	\]
	and the doping profile is allowed to be non-flat. The purpose of the
	present paper is to fill this gap.
	
	
	Our analysis is inspired by the large-data theory developed for the
	compressible Navier--Stokes equations in
	\cite{VK1995,HL2016,FLL2022}. Nevertheless, the Poisson
	coupling and the general boundary geometry create several new
	difficulties. We describe the main points below.
	
	First, the electrostatic force $\rho\nabla\Phi$ enters directly into
	the effective viscous flux equation. It is therefore convenient to
	introduce the modified material derivative
	\[
	\widetilde u:=\dot u-\nabla\Phi := \pl_t u + (u \cdot \nabla) u - \nabla \Phi.
	\]
	With
	\[
	F=(2\mu+\lambda(\rho))\mbox{div} u-P,
	\qquad
	\omega=\mbox{curl} u,
	\]
	the momentum equation can be written in the form
	\[
	\rho\widetilde u=\nabla F+\mu\nabla^\perp\omega.
	\]
	Although this decomposition resembles the one for the usual
	compressible Navier--Stokes system, the Poisson force produces
	additional terms at several stages of the argument. In particular,
	the pointwise representation of $F$ contains an additional electric
	contribution involving $\rho\nabla\Phi$, and the lower- and
	higher-order estimates have to be reformulated in terms of
	$\widetilde u$. The Neumann condition for the electric potential is
	also coupled with the Navier-slip boundary condition in the estimates
	of the boundary terms.
	
	Second, an important ingredient of our proof is a self-contained
	derivation of the $L^p$ estimate for the density in Lemma \ref{lem2}.
	This issue is particularly delicate in a general curved bounded
	domain. To explain the difficulty, set
	\[
	G=\rho u\otimes u.
	\]
	The natural elliptic problem associated with
	$\mbox{div} \mbox{div} G$ is
	\[
	\Delta\eta=\mbox{div} \mbox{div} G,
	\qquad
	\partial_n\eta=(\mbox{div} G)\cdot n,
	\]
	and the normal derivative is in general non-zero on a curved
	boundary. Thus the corresponding elliptic operator cannot simply be
	treated as a homogeneous Neumann inverse.
	
	To overcome this difficulty, we introduce the zero-mean Neumann
	operator $\mathcal N$ by
	\[
	\Delta\mathcal Nf=f-\overline f,
	\qquad
	\partial_n\mathcal Nf=0,
	\qquad
	\int_\Omega\mathcal Nf\,dx=0,
	\]
	and use a duality argument to obtain the quantitative estimate
	\[
	\|\eta\|_{L^p}
	\le Cp\|\rho|u|^2\|_{L^p}.
	\]
	At the same time, the difference between the natural normal derivative
	of $\eta$ and that of the effective viscous flux gives rise to a
	non-trivial harmonic correction. More precisely, we introduce the
	zero-mean harmonic function $h$ satisfying
	\[
	\Delta h=0,
	\qquad
	\partial_nh
	=\mu\partial_\tau\big(A(u\cdot\tau)\big),
	\]
	and prove the estimate
	\[
	\|h\|_{L^p}
	\le C\sqrt p\,\|u\|_{H^1}.
	\]
	This construction allows us to recover the density estimate without
	imposing an artificial homogeneous boundary condition on the
	auxiliary elliptic problem. Together with a further auxiliary
	potential associated with the Poisson force, this yields the
	polynomial-in-$p$ estimate required for the subsequent density
	upper bound.
	
	Third, to obtain the uniform upper bound of the density, we adapt the
	pull-back Green function and conformal-mapping argument of
	\cite{FLL2022} to the present Navier--Stokes--Poisson system.
	Besides the commutator term already present in the Navier--Stokes
	equations, the representation of the effective viscous flux now
	contains the additional term
	\[
	J_2(x)
	=
	\int_\Omega
	\nabla_y\widetilde N(x,y)\cdot
	\rho(y)\nabla\Phi(y)\,dy.
	\]
	Moreover, for the general Navier-slip condition
	\[
	\mbox{curl} u=-A u\cdot n^\perp
	\quad\text{on }\partial\Omega,
	\]
	the Neumann boundary condition for $F$ contains the tangential
	boundary contribution
	\[
	-\mu\partial_\tau\big(Au_\tau\big).
	\]
	Both contributions have to be retained in the Green representation
	and estimated together with the usual commutator terms. After combining
	the elliptic estimates for the Poisson equation, the higher
	integrability of the density and the boundary estimates, we obtain
	the crucial inequality
	\[
	R_T^\beta
	\le
	C(\varepsilon)
	R_T^{\max\left\{
		1+\frac{\beta}{4}+3\varepsilon,
		\frac{2+\beta}{3},
		1+\varepsilon
		\right\}}.
	\]
	Since $\beta>4/3$, choosing $\varepsilon>0$ sufficiently small closes
	the estimate and gives a uniform upper bound for the density.
	
	Finally, the Poisson coupling also requires some modifications in the
	continuation and compactness arguments. In particular, the
	approximation of the initial density has to preserve the Neumann
	compatibility condition
	\[
	\int_\Omega\rho_0\,dx
	=
	\int_\Omega b(x)\,dx.
	\]
	We therefore use charge-preserving positive approximations and exploit
	the continuity of the Neumann solution operator to pass to the limit
	in the nonlinear electric force $\rho\nabla\Phi$. With these
	additional ingredients, the lower- and higher-order estimates can be
	closed globally in time, and the global strong and weak solutions
	stated in Theorem \ref{thm1} and Theorem \ref{thm2} follow.

		Before stating the main results, we introduce the notation and conventions used throughout this paper. For a positive integer $k$ and $1 \leq p < \infty$, we use the following notation for the standard $L^p$ spaces and Sobolev spaces:
	\begin{equation}
		\left\{
		\begin{array}{lll}
			L^p=L^p(\Omega), \quad W^{k,p}=W^{k,p}(\Omega), \quad H^k=W^{k,2}(\Omega),\\[2mm]
			\|f\|_{L^p}= \|f\|_{L^p(\Omega)}, \quad \|f\|_{W^{k,p}}=\|f\|_{W^{k,p}(\Omega)}, \quad \|f\|_{H^k}=\|f\|_{W^{k,2}(\Omega)},\\[2mm]
			\tilde{H}^1= \{ u \in H^1(\Omega) |  u\cdot n=0,\quad \mbox{curl}u = -A u \cdot n^{\bot},\quad \mbox{on}\quad \partial\Omega\}, \\[2mm]
			\tilde{W}^{2,q} = \{ u \in W^{2,q}(\Omega) | u \cdot n =0, \quad \mbox{on} \quad \pl \Omega \}.
		\end{array}
		\right.
	\end{equation}
	
	The material derivative and the transpose gradient are given by
	$$\frac{D}{Dt} f = \dot{f} := \frac{\pl}{\pl_t}f + u \cdot \nabla f, \quad \nabla^{\bot}= ( \pl_2, -\pl_1).$$

	We first define weak and strong solutions as follows:
	
	\begin{defn}\label{wandssolution}
		Let $T>0$ be a finite constant. A solution $(\rho,u, \Phi)$ to \eqref{origin} is called a weak solution if it satisfies \eqref{origin} in the sense of distributions. Moreover, if
		all the derivatives involved in \eqref{origin} are regular distributions and \eqref{origin} holds almost everywhere in $\Omega \times (0,T)$, the solution is called a strong solution.
	\end{defn}

	We state the main result concerning the global existence of strong solutions as follows:

	\begin{thm}\label{thm1}
		Let $\Omega$ be a simply connected bounded domain in $\mathbb{R}^2$ with $C^{2,1}$ boundary $\pl \Omega$. Assume that
		\begin{align}\label{betagamma}
			\beta > \frac{4}{3}, \quad \gamma >1,
		\end{align}
		and that the initial data $\rho_0 \geq 0$, $m_0$ and $\nabla \Phi_0$ satisfy, for some $q>2$,
		\begin{align}\label{initialcondition}
			\rho_0 \in W^{1,q}, \quad u_0 \in \tilde{H}^1, \quad \nabla \Phi_0 \in \tilde{W}^{2,q}, \quad m_0=\rho_0 u_0, \quad \Delta\Phi_0=\rho_0-b(x),
		\end{align}
		with $b(x) \in W^{1,q}$.
		
		Then the problem \eqref{origin}-\eqref{boundary2} has a unique strong solution $(\rho ,u,\nabla \Phi)$ in $\Omega \times (0,\infty)$ satisfying, for any $0<T< \infty$,
		\begin{equation}\label{theorem1}
			\left\{
			\begin{array}{lll}
				\rho \in C([0,T]; W^{1,q}), \quad \rho_{t} \in L^{\infty}(0,T;L^2),\\[2mm]
				\nabla \Phi \in C([0,T]; W^{2,q}), \quad \nabla \Phi_{t} \in L^{\infty}(0,T;H^1),\\[2mm]
				u \in L^{\infty}(0,T;H^1) \cap L^{1+\frac{1}{q}}(0,T; W^{2,q}),\\[2mm]
				\sqrt{t}u \in L^{\infty}(0,T;H^2) \cap L^{2}(0,T; W^{2,q}), \quad \sqrt{t}u_t \in L^2(0,T;H^1),\\[2mm]
				\rho u \in C([0,T];L^2), \quad \sqrt{\rho}u_t \in L^2((0,T)\times \Omega).
			\end{array}
			\right.
		\end{equation}

	\end{thm}

	The second result concerns the global existence of weak solutions.
	
	\begin{thm}\label{thm2}
		Assume the conditions of Theorem \ref{thm1}, except that the requirements $\rho_0 \in W^{1,q}$ and $\nabla \Phi_0 \in \tilde{W}^{2,q}$ in \eqref{initialcondition} are replaced by $\rho_0 \in L^{\infty}$, and let $\Phi_0$ satisfy $\Delta \Phi_0 = \rho_0 - b(x)$ with $(\pl_n \Phi_0)|_{\pl \Omega}=0$. Then there exists at least one weak solution $(\rho ,u, \nabla \Phi)$ to the problem \eqref{origin}-\eqref{boundary2} in $\Omega \times (0,\infty)$ satisfying, for any $0<\tau \leq T< \infty$ and $p \geq 1$,
		\begin{equation}\label{theorem2}
			\left\{
			\begin{array}{lll}
				\rho \in L^{\infty}(0,T;L^{\infty}) \cap C([0,T];L^p), \quad \nabla \Phi \in L^{\infty}(0,T;L^{\infty}) \cap C([0,T];W^{1,p}) \\[2mm]
				u \in L^{\infty}(0,T;H^1), \quad u_t \in L^2(\tau,T; L^{2}), \quad \nabla u \in L^{\infty}(\tau,T; L^{p})\\[2mm]
				\mbox{curl} u, (2\mu + \lambda)\mbox{div}u - P \in  L^{2}(0,T;H^1) \cap L^{\frac{3}{2}}(0,T; L^{\infty}).
			\end{array}
			\right.
		\end{equation}

		Moreover, there exists a positive constant $C$ depending on $\Omega,T,\mu,\beta, \gamma,\|b(x)\|_{W^{1,q}}, \|\rho_0\|_{L^{\infty}}$, and $\|u_0\|_{H^1}$ such that
		\begin{equation}\label{infrho}
			\inf_{(x,t)\in \Omega \times (0,T)} \rho(x,t) \geq C^{-1} \inf_{x\in \Omega } \rho_0(x).
		\end{equation}
		
	\end{thm}

	We now briefly describe the main strategy for obtaining the uniform upper
	bound of the density. Following the approach developed by Fan, Li and Li
	\cite{FLL2022}, we introduce
	\[
	\theta(\rho)
	:=2\mu\log\rho+\frac{1}{\beta}\rho^\beta
	\]
	and the effective viscous flux
	\[
	F:=(2\mu+\lambda(\rho))\mbox{div} u-P.
	\]
	By the continuity equation,
	\[
	\frac{D}{Dt}\theta(\rho)+P=-F.
	\]
	Hence, the upper bound of $\rho$ is reduced to obtaining a suitable
	pointwise estimate for $F$. In the present Navier--Stokes--Poisson system,
	$F$ satisfies
	\[
	\Delta F=\mbox{div}\big(\rho(\dot u-\nabla\Phi)\big)
	\]
	together with the corresponding nonhomogeneous Neumann boundary condition
	induced by the Navier-slip boundary condition.
	
	To estimate $F$ in a general bounded simply connected domain, we follow
	the conformal-mapping and Green-function argument introduced in
	\cite{FLL2022}. More precisely, the Neumann Green function on the unit
	disc is pulled back to $\Omega$ through a conformal mapping. The
	preservation of angles under conformal mappings, together with the
	slip boundary condition, allows the singular part of the resulting
	integral representation to be reduced to a commutator-type term.
	This provides the basic mechanism for deriving the uniform upper bound
	of the density. Compared with \cite{FLL2022}, the Poisson coupling generates additional interior and
	boundary terms, whose estimates constitute part of the main new	difficulties of the present paper.
	
	The remainder of the paper is organized as follows.
	Section \ref{section2} collects the preliminary results.
	In Section \ref{section3}, we derive the key a priori estimates leading
	to the uniform upper bound of the density.
	Section \ref{section4} is devoted to the lower- and higher-order estimates.
	Finally, in Section \ref{section5}, we prove Theorems \ref{thm1} and
	\ref{thm2} by approximation, continuation, and compactness arguments.

	\section{Preliminaries}\label{section2}
	
	We first recall the standard local existence theory for initial densities bounded away from vacuum, which forms the basis of our proof of global existence; see \cite{SS1993,S1980}.
	
	\begin{lem}\label{lem-localexistence}
		Assume that $(\rho_0,m_0,\Phi_0)$ satisfies
		\begin{align}\label{initialvalue}
			\rho_0 \in H^{2}, \quad \inf_{x\in \Omega }\rho_0(x)>0, \quad u_0 \in H^2 \cap \tilde{H}^1, \quad \nabla \Phi_0 \in H^3 \cap \tilde{W}^{2,q}, \quad m_0=\rho_0 u_0, \quad \Delta \Phi_0=\rho_0-b(x),
		\end{align}
		with $b(x) \in H^2$.
		
		Then there exist a sufficiently small time $T>0$ and a constant $C_0>0$, both depending only on $\Omega,\mu,\beta,\gamma,\|\rho_0\|_{H^2}$, $\|b(x)\|_{H^2}$, $\|u_0\|_{H^2}$, and $\inf_{x\in \Omega }\rho_0(x)$, such that the problem \eqref{origin}-\eqref{boundary2} has a unique strong solution $(\rho ,u,\nabla \Phi)$ in $\Omega \times (0, T)$ satisfying
		\begin{equation}\label{localexistence}
			\left\{
			\begin{array}{lll}
				\rho \in C([0,T]; H^2), \quad \rho_{t} \in C([0,T];H^1),\\[2mm]
				\nabla \Phi \in C([0,T]; H^3), \quad \nabla \Phi_{t} \in C([0,T];H^2),\\[2mm]
				u \in C([0,T];H^2)\cap L^{2}(0,T;H^3) , \quad u_t \in  C([0,T]; L^2) \cap L^{2}(0,T;H^1),\\[2mm]
				\rho u \in C([0,T];L^2), \quad \sqrt{\rho}u_t \in L^2((0,T)\times \Omega).
			\end{array}
			\right.
		\end{equation}
		and
		\begin{equation}\label{localinf}
			\inf_{(x,t)\in \Omega \times (0,T)} \rho(x,t) \geq C_0 >0.
		\end{equation}
	\end{lem}

	The following lemma controls the $H^{m}$-norm of a vector-valued function $v$ in terms of the $H^{m-1}$-norms of $\mbox{curl}v$ and $\mbox{div} v$ (see \cite{XX2007}).
	\begin{lem}\label{lem-div-curl}
		Let $\Omega$ be a domain in $\mathbb{R}^{2}$ with smooth boundary $\partial\Omega$ and outward normal $n$. Then there exists a constant $C>0$ such that
		\begin{equation}\label{A3-1}
			\|v\|_{H^{s}}\leq C\left(\|\mbox{div}v\|_{H^{s-1}}+\|\mbox{curl}v\|_{H^{s-1}}
			+ |v\cdot n|_{H^{s-\frac{1}{2}}(\partial\Omega)}+\|v\|_{L^2}\right),
		\end{equation}
		for any $v\in H^{s}(\Omega)$, $s\geq1$.
	\end{lem}
	
	\begin{lem} ([\cite{wahl}, Theorem 3.2] and \cite{A2014, M2005, wahl})\label{lem-important}
		Let $1<p<\infty$ and $\Omega$ be a simply connected bounded domain in $\mathbb{R}^2$ with Lipschitz boundary $\pl \Omega$. Then there exists a positive constant $C$ depending only on $\Omega$ and $p$ such that every $v \in W^{1,p}$ with $v \cdot n=0$ on $\pl \Omega$ satisfies
		\begin{equation}\label{important}
			\|\nabla v\|_{L^p}\leq C\left(\|\mbox{div}v\|_{L^p}+\|\mbox{curl}v\|_{L^p}\right).
		\end{equation}
	\end{lem}
	
	We will also frequently use the following modified version of the classical Poincar\'{e}--Sobolev inequality.
	
	\begin{lem}(see \cite{T1976})\label{lem-poincare}
		There exists a positive constant $C$ depending only on $\Omega$ such that every function $u \in H^1(\Omega)$ satisfies, for $2<p<\infty$,
		\begin{equation}\label{poincare}
			\|u\|_{L^p} \leq C p^{\frac{1}{2}} \|u\|^{\frac{2}{p}}_{L^2} \|u\|^{1-\frac{2}{p}}_{H^1},
		\end{equation}
		where $\|u\|_{H^1}$ can be replaced by $\|\nabla u\|_{L^2}$ provided that
		$$(u\cdot n)|_{\pl \Omega}=0 \quad \mbox{or} \quad \int_{\Omega} u \mbox{d} x=0.$$
	\end{lem}
	
	We now recall the Poincar\'{e} inequality for tangential vector fields. Its proof uses a standard contradiction argument and can be found in \cite{MM2015}.
	
	\begin{lem}\label{lem-poincare2}
		Let $\Omega \in \mathbb{R}^n$ be an open, bounded Lipschitz set. Then every $u \in W^{1,2}(\Omega)$ satisfying $u \cdot n=0$ on $\pl \Omega$ obeys
		\begin{equation}\label{poincare2}
			\|u\|_{L^2(\Omega)} \leq C(\Omega) \|\nabla u\|_{L^2},
		\end{equation}
		where the constant $C(\Omega)$ depends only on $\Omega$ (it is independent of $u$).
	\end{lem}

	\begin{prop}(see \cite{CL2023})\label{lem-boundary-key}
		Assume that $\Omega$ is a bounded $C^{2,1}$ domain in $\mathbb{R}^2$. Let $v, f \in H^{1}$ with $v\cdot n=0$ on $\partial\Omega$. Then
		\begin{align}
			\left|\int_{\partial\Omega}v\cdot\nabla f\right|\leq C\|v\|_{H^1}\| f\|_{H^1}.\nonumber
		\end{align}
	\end{prop}
	
	Next, to estimate the $L^{\infty}(0,T;L^p)$-norm of $\nabla \rho$, we use the following inequality of Beale-Kato-Majda type, first proved in \cite{BKM1984} for $\mbox{div}u=0$.
	
	\begin{lem}(see \cite{CL2023, BKM1984})\label{lem-nablau}
		For $2<q<\infty$, there is a constant $C$ depending only on $\Omega$ and $q$ such that every $u \in \{ W^{2,q}(\Omega)| u \cdot n=0, \mbox{curl} u = -Au \cdot n^{\bot} \quad on \quad \pl \Omega \}$ satisfies
		\begin{equation}\label{nablerho}
			\|\nabla u \|_{L^{\infty}} \leq C( \|\mbox{div}u\|_{L^{\infty}} + \|\mbox{curl} u\|_{L^{\infty}} )\log (e+ \|\nabla^2 u\|_{L^q}) + C \|\nabla u\|_{L^2} +C.
		\end{equation}
	\end{lem}
	
	\begin{lem}\label{theorem-trace}
		Let $f \in H^{m_1} (\Omega) \cap H^{m_2} (\Omega)$ and $\triangledown f \in H^{m_2} (\Omega)$. Then
		
		\begin{equation}\label{2eq2-3}
			|f|^2_{H^s(\pl \Omega)} \leq C(\|\nabla f\|_{H^{m_2}} + \|f\|_{H^{m_2}}) \| f\|_{H^{m_1}},
		\end{equation}
		where $m_1+ m_2 \geq 2s \geq 0$ and $m_1$, $m_2$ are nonnegative integers.
	\end{lem}

	Next, let $\Omega$ be as in Theorem \ref{thm1} and let $\mathbb{D}$ denote the unit disc. By the Riemann mapping theorem (see \cite{SS2003}, Chapter 9), there exists a conformal mapping $\psi = (\psi_1,\psi_2): \bar{\Omega} \to \bar{\mathbb{D}}$ whose components $\psi_1$ and $\psi_2$ satisfy the Cauchy-Riemann equations:
	\begin{equation}\label{psi_1}
		\left\{
		\begin{array}{lll}
			\pl_1 \psi_1=\pl_2 \psi_2\\[2mm]
			\pl_2 \psi_1=- \pl_1 \psi_2.
		\end{array}
		\right.
	\end{equation}\\
	
	The conformal mapping $\psi(x)$ also has the following properties:

	\begin{lem} (see \cite{SS2003,W1935,W1961} )\label{lem-cauchy-riemann}
		The conformal mapping $\psi = (\psi_1,\psi_2): \bar{\Omega} \to \bar{\mathbb{D}}$ is smooth and has the following properties:
		\begin{itemize}
			\item $\psi(x)$ is a one-to-one holomorphic mapping from $\Omega$ to $\mathbb{D}$ and maps the boundary $\pl \Omega$ onto the boundary $\pl \mathbb{D}$.
			\item For any integer $k=0,1,2$, there exists some constant $C$ depending only on $\Omega$ and $k$ such that
			$$|\nabla^k \psi(x)|\leq C, \quad \forall x \in \bar{\Omega}.$$
			Consequently,
			$$|\nabla^k \psi(x) - \nabla^k \psi(y)|\leq C|x-y|, \quad \forall x,y \in \bar{\Omega}.$$
			\item There exist two positive constants $c_1,c_2$ such that
			$$c_1|x-y|\leq |\psi(x)-\psi(y)|\leq c_2|x-y|.$$
			\item The mapping $\psi$ preserves angles. More precisely, for any
			$x\in\overline{\Omega}$ and $v,w\in\mathbb R^2$,
			\[
			\langle D\psi(x)v,D\psi(x)w\rangle
			=
			|\psi'(x)|^2\langle v,w\rangle.
			\]
			Consequently,
			\[
			v\perp w
			\quad\Longrightarrow\quad
			D\psi(x)v\perp D\psi(x)w.
			\]
			\item For any harmonic function $h(y)$ in $\mathbb{D}$, the composition $h(\psi(x))$ is harmonic in $\Omega$; that is,
			$$\Delta h(\psi(x))=0 \quad \mbox{in} \quad \Omega,$$
			provided
			$$\Delta h(y)=0 \quad \mbox{in} \quad \mathbb{D}.$$
		\end{itemize}
		
	\end{lem}
	
	Finally, the preservation of angles under conformal mappings immediately yields the following result:
	
	\begin{lem}(\cite{FLL2022})\label{lem-conformal}
		Let $n$ be the outward unit normal vector to $\pl \Omega$ at $x_0$. Then
		\begin{equation}\label{conformal}
			n \cdot \nabla \psi(x_0)=|\nabla \psi_1(x_0)|n,
		\end{equation}
		where $n$ is the outward unit normal vector to $\pl \mathbb{D}$ at $\psi(x_0)$.
	\end{lem}

	\begin{rmk}\label{rmk-boundary}
		The boundary condition $u \cdot n=0$ on $\pl \Omega$ implies that, for any $y_0 \in \pl \Omega$, the vectors $\psi(y_0)$ and $(u \cdot \nabla)\psi(y_0)$ are, respectively, an outward normal and a tangent vector to $\pl \mathbb{D}$ at $\psi(y_0)$. Consequently,
		$$u_i(y_0 ) \cdot \pl_i \psi_j(y_0) \cdot \psi_j(y_0)= \langle(u \cdot \nabla)\psi(y_0), \psi(y_0) \rangle=0,$$
		which is an important observation for the subsequent analysis.
	\end{rmk}

	\section{A Priori Estimates (I): Upper Bound for the Density}\label{section3}
	
	In this section and the next, we assume that $(\rho,u, \nabla \Phi)$ is the strong solution to \eqref{origin}-\eqref{boundary2} on $\Omega \times (0,T)$ provided by Lemma \ref{lem-localexistence}.
	We define the effective viscous flux $F$ and the vorticity $\omega$ in the usual manner:
	\begin{equation}\label{Fdefinition}
		F:= (2\mu + \lambda(\rho))\mbox{div}u -P,
	\end{equation}
	and
	\begin{equation}\label{omegadefinition}
		\omega:=\nabla^{\bot} \cdot u= \pl_2 u_1 -\pl_1 u_2.
	\end{equation}
	
	Then, we set
	\begin{equation}\label{A_1}
		A_1^2(t):= 1 + \int_{\Omega}( \omega^2(t) + \frac{F^2(t)}{2\mu + \lambda(\rho)}) \mbox{d} x + \int_{\partial\Omega} A |u|^2 \mbox{d}s,
	\end{equation}
	\begin{equation}\label{A_2}
		A_2^2(t):=\int_{\Omega} \rho(t) |\tilde{u}(t)|^2 \mbox{d}x:=\int_{\Omega}\rho(t) |\dot{u}(t)-\nabla \Phi(t)|^2\mbox{d} x,
	\end{equation}
	and
	\begin{equation}\label{R_T}
		R_T:= 1+ \sup_{0 \leq t \leq T} \|\rho (t)\|_{L^{\infty}}.
	\end{equation}
	
	We now state the standard energy estimate.
	
	\begin{lem}\label{lem1}
		There exists a positive constant $C$ depending only on $\gamma$, $\|\rho_0\|_{L^{\gamma} \cap L^2}$, $\|b(x)\|_{L^2}$, and $\|\sqrt{\rho_0} u_0\|_{L^2}$ such that
		\begin{align}\label{lem1-1}
			& \sup_{0 \leq t \leq T} \int_{\Omega} ( \rho |u|^2 + \rho^{\gamma} + |\nabla \Phi|^2 ) \mbox{d}x + \int_{0}^{T} \int_{\Omega}( (2 \mu + \lambda)|\mbox{div}u|^2+ \mu |\omega|^2 )\mbox{d}x \mbox{d}t \leq C.
		\end{align}
	\end{lem}
	
	\begin{proof}
		A direct calculation shows that
		\begin{equation}\label{1-1}
			P_t + \mbox{div}(Pu) + (\gamma-1) P \mbox{div}u=0,
		\end{equation}
	by $\eqref{origin}^1$. Integrating with respect to $x$ and then integrating by parts, using the boundary condition \eqref{boundary1}, namely $(u \cdot n)|_{\pl \Omega}=0$, we obtain
	\begin{equation}\label{1-2}
		\frac{d}{dt} \int_{\Omega} \frac{\rho^{\gamma}}{\gamma-1} \mbox{d}x + \int_{\Omega} P\mbox{div}u\mbox{d}x =0.
	\end{equation}
	
	Recall that
	$$\rho u_t + \rho (u \cdot \nabla )u + \nabla P= \nabla ( (2\mu + \lambda(\rho)) \mbox{div}u )+ \mu \nabla^{\bot} \omega + \rho \nabla \Phi.$$
	
	Taking the inner product with $u$ and integrating with respect to $x$, we obtain
	\begin{align}\label{1-3}
		&\int_{\Omega}( \rho u_t \cdot u + \rho (u \cdot \nabla )u \cdot u+ \nabla P\cdot u) \mbox{d}x = \int_{\Omega}(\nabla ( (2\mu + \lambda(\rho)) \mbox{div}u )\cdot u+ \mu \nabla^{\bot} \omega \cdot u + \rho \nabla \Phi \cdot u) \mbox{d}x .
	\end{align}
	
	We have the vector calculus identities
	\begin{equation}\label{vector-1}
		\nabla \cdot (\mbox{div}u u)= |\mbox{div}u|^2+ \nabla \mbox{div}u \cdot u,
	\end{equation}
	and
	\begin{equation}\label{vector-2}
		\nabla^{\bot } \cdot (\omega u  )= |\omega|^2 + \nabla^{\bot} \omega \cdot u.
	\end{equation}
	
	Integrating by parts, applying the boundary conditions, and using $\eqref{origin}^1$ and \eqref{1-2} in \eqref{1-3}, we obtain
	\begin{align}\label{1-4}
		&\int_{\Omega}( \frac{\pl }{\pl t} (\frac{1}{2} \rho |u|^2)+ \frac{\pl}{\pl t}\frac{\rho^{\gamma}}{\gamma-1}) \mbox{d}x = \int_{\Omega}(\nabla ( (2\mu + \lambda(\rho)) \mbox{div}u )\cdot u+ \mu \nabla^{\bot} \omega \cdot u + \rho \nabla \Phi \cdot u) \mbox{d}x \nonumber\\[2mm]
		& = \int_{\Omega}(- (2\mu + \lambda(\rho)) |\mbox{div}u|^2 - \mu | \omega|^2 - \Phi\mbox{div}(\rho u)) \mbox{d}x - \mu \int_{\partial\Omega} A|u|^2 \mbox{d}s \nonumber\\[2mm]
		& = \int_{\Omega}(- (2\mu + \lambda(\rho)) |\mbox{div}u|^2 - \mu | \omega|^2 + \Phi \Delta\Phi_t) \mbox{d}x - \mu \int_{\partial\Omega} A|u|^2 \mbox{d}s \nonumber\\[2mm]
		& = \int_{\Omega}(- (2\mu + \lambda(\rho)) |\mbox{div}u|^2 - \mu | \omega|^2 -\frac{\pl}{\pl t}\frac{1}{2}|\nabla \Phi|^2) \mbox{d}x- \mu \int_{\partial\Omega} A|u|^2 \mbox{d}s ,
	\end{align}
	where we have used $\eqref{origin}^1$ and $\eqref{origin}^3$ to obtain $\mbox{div}(\rho u)=-\rho_t= (-b(x)-\Delta \Phi)_t= -\Delta\Phi_t$, together with the boundary conditions \eqref{boundary1}, with $A$ positive on the boundary, and $(\nabla \Phi \cdot n)|_{\pl \Omega}=0$.
	
	Integrating with respect to $t$ completes the proof.
	
    \end{proof}
    
    Next, we establish estimates for the $L^{\infty}(0,T;L^p{\Omega})$ norm of the density and the electric potential. The proof is similar to the argument in \cite{VK1995}.
    
    \begin{lem}\label{lem2}
    	Let $\beta > \frac{4}{3}$. Then, for any $1 \leq p < \infty$, there is a constant $C$ depending on $\Omega$, $T$, $\mu$, $\beta$, $\gamma$, $\|b(x)\|_{L^{\infty}}$, $\|\rho_0\|_{L^{\infty}}$, and $\|u_0\|_{H^1}$ such that
    	\begin{equation}\label{lem2-rho}
    		\sup_{0 \leq t \leq T}\|\rho(t)\|_{L^p} \leq C(T) p^{\frac{2}{\beta-1}},
    	\end{equation}
    	and
    	\begin{equation}\label{lem2-phi}
    		\sup_{0 \leq t \leq T}\|\nabla \Phi(t) \|_{L^p} \leq C(T) p^{\frac{2}{\beta-1}},
    	\end{equation}
    	Here and below, $C$ denotes a generic positive constant depending on $\Omega$, $T$, $\mu$, $\beta$, $\gamma$, $\|b(x)\|_{L^{\infty}}$, $\|\rho_0\|_{L^{\infty}}$, and $\|u_0\|_{H^1}$; we write $C(\alpha)$ to indicate that $C$ also depends on $\alpha$.
    \end{lem}
    
    \begin{proof}
    	
    	We first introduce some quantitative results for the Neumann problem.
    	Write $\av f=|\Om|^{-1}\int_\Om f$ and define $\GN f$ by
    	\begin{equation}\label{neumann-inverse}
    		\Delta\GN f=f-\av f,\qquad \partial_n\GN f=0,
    		\qquad\int_\Om\GN f \mbox{d}x=0.
    	\end{equation}
    	Thus this inverse is always applied after removal of the constant mode.
    	We first justify existence and uniqueness and explain the role of the subtracted average. For $1<q<\infty$, set
    	\[
    	V=\left\{v\in H^1(\Om):\int_\Om v \mbox{d}x=0\right\},
    	\qquad \norm v{V}:=\norm{\nabla v}{L^2}.
    	\]
    	
    	Poincar\'e's inequality shows that this norm is equivalent to the $H^1$ norm. The weak formulation for $w=\GN f$ is
    	\[
    	\int_\Om\nabla w\cdot\nabla v \mbox{d}x
    	=-\int_\Om(f-\av f)v \mbox{d}x,\qquad v\in V.
    	\]
    	The right-hand side is bounded on $V$: in two dimensions
    	$H^1(\Om)\hookrightarrow L^{q'}(\Om)$ for every finite
    	$q'=q/(q-1)$, and
    	$\norm{f-\av f}{L^q}\le2\norm f{L^q}$.
    	The bilinear form on the left is coercive on $V$. The Lax--Milgram theorem therefore gives a unique $w\in V$. Since $\int_\Om(f-\av f)=0$, subtracting the mean of an arbitrary test function extends this identity to every $v\in H^1(\Om)$. This is precisely the homogeneous Neumann boundary condition in the weak sense. Standard Neumann regularity on this domain gives $w\in W^{2,q}(\Om)$. Alternatively, uniqueness follows by testing the difference of two solutions with itself: its gradient vanishes, and its zero average eliminates the remaining constant.
    	Thus \eqref{neumann-inverse} is well defined for every $f\in L^q$;
    	it is not an inverse of $\Delta$ on constants, since $\GN c=0$.
    	
    	We shall use the following quantitative estimates for the Neumann Laplacian. They follow from the Calder\'on--Zygmund theory for singular	integral operators and the standard estimates for the two-dimensional
    	Neumann Green function; see \cite{CZ1952,TKB2013,Stein1970}. Here the explicit dependence on $a$ in \eqref{hessian} and	\eqref{projection} follows from the weak type $(1,1)$ estimate and the $L^2$ boundedness of the corresponding Calder\'on--Zygmund operators, together with the Marcinkiewicz interpolation theorem and duality.
    	
    	\begin{lem}\label{elliptic}
    		
    		There exists a constant $C_\Om$ depending only on $\Omega$ such that:
    		
    		\begin{enumerate}
    			\item If $1<a\le2$, then
    			\begin{equation}\label{hessian}
    				\norm{D^2\GN f}{L^a}\le {C_\Om\over a-1}\norm f{L^a}.
    			\end{equation}
    			\item If $z$ is the zero-mean weak solution of
    			\[
    			\int_\Om\nabla z\cdot\nabla\varphi
    			=\int_\Om v\cdot\nabla\varphi\qquad(\varphi\in H^1(\Om)),
    			\]
    			then for $2\le a<\infty$,
    			\begin{equation}\label{projection}
    				\norm{\nabla z}{L^a}\le C_\Om a\norm v{L^a}.
    			\end{equation}
    			Equivalently, this corresponds to
    			\[
    			\Delta z=\Div v,\qquad \partial_n z=v\cdot n
    			\]
    			in the weak sense.
    			\item The zero-mean Neumann Green function, with sign convention
    			$\Delta_yN(x,y)=\delta_x(y)-|\Om|^{-1}$, satisfies
    			\begin{equation}\label{green-gradient}
    				|\nabla_yN(x,y)|\le {C_\Om\over |x-y|}\quad(x\ne y).
    			\end{equation}
    		\end{enumerate}
    		
    	\end{lem}

    	Recalling the definitions $F=(2\mu + \lambda(\rho))\mbox{div}u -P$ and $\omega=\pl_2 u_1- \pl_1 u_2$, we have
    	\begin{equation}\label{2-1}
    		\pl_t ( \mbox{div}(\rho u) ) + ( \pl_{11}(\rho u_1^2) + 2 \pl_1 \pl_2 (\rho u_1 u_2) + \pl_{22}( \rho u_2^2 ) ) = \Delta F + \rho \Delta \Phi + \nabla \rho \cdot \nabla \Phi.
    	\end{equation}
    	
    	Let $\xi$ solve
    	\begin{equation}\label{xi}
    		\Delta\xi=\mbox{div}(\rho u),\qquad (\partial_n\xi)|_{\pl \Omega}=0,
    		\qquad \int_{\Omega} \xi \mbox{d}x=0.
    	\end{equation}
    	Compatibility follows from $\rho (u\cdot n)=0$ on the boundary $\pl \Omega$. Define $\eta$ by
    	\begin{equation}\label{eta}
    		\Delta\eta=\mbox{div}\mbox{div} G:= \mbox{div}\mbox{div}(\rho u \otimes u),\qquad
    		\partial_n\eta|_{\pl \Omega}=((\mbox{div} G)\cdot n)|_{\pl \Omega},
    		\qquad\int_{\Omega}\eta \mbox{d}x=0.
    	\end{equation}
    	This Neumann problem is compatible by the divergence theorem. Since $u$
    	is tangent to the boundary,
    	\begin{align}
    		(\mbox{div} G)\cdot n
    		&=\rho((u\cdot\nabla)u)\cdot n
    		=-\rho\,u\cdot((u\cdot\nabla)n)
    		\quad\text{on }\partial\Omega.\label{curvature}
    	\end{align}
    	In particular, this datum is not in general zero.
    	
    	For completeness, the $L^a$ construction of \eqref{eta} can be carried out without estimating any boundary trace of $G$. For smooth $G$ with $Gn|_{\pl \Omega} = ((\rho u) u \cdot n)|_{\pl \Omega}=0$ and any smooth $\varphi$ with $\partial_n\varphi=0$, two integrations by parts give
    	\begin{align}
    		\int_{\Omega}\eta\Delta\varphi \mbox{d}x
    		&= \int_{\pl \Omega} \nabla \varphi \cdot n \eta \mbox{d}s - \int_{\Omega} \mbox{div} G \cdot \nabla \varphi \mbox{d}x =\int_{\Omega}\varphi\mbox{div}\mbox{div} G \mbox{d}x
    		-\int_{\partial\Omega}\varphi(\mbox{div} G)\cdot n \mbox{d}s\notag\\
    		& = -\int_{\Omega}\nabla\varphi\cdot\mbox{div} G \mbox{d}x + \int_{\partial\Omega} \varphi \mbox{div} G \cdot n \mbox{d}s-\int_{\partial\Omega}\varphi(\mbox{div} G)\cdot n \mbox{d}s\notag\\ 
    		&=\int_{\Omega}D^2\varphi : G \mbox{d}x - \int_{\partial\Omega} (Gn) \cdot \nabla \varphi \mbox{d}s
    		=\int_{\Omega} G:D^2\varphi \mbox{d}x.\label{eta-duality}
    	\end{align}
    	
    	If $a\ge2$, $a'=a/(a-1)$, and $f\in L^{a'}$, take
    	$\varphi=\mathcal{N} f$. The subtracted average must first be retained:
    	\begin{align*}
    		\int_{\Omega}\eta(f- \bar{f}) \mbox{d}x
    		&=\int_{\Omega}\eta\Delta\mathcal{N} f \mbox{d}x
    		=\int_{\Omega} G:D^2\mathcal{N} f \mbox{d}x,\\
    		\int_{\Omega}\eta(f-\bar{f}) \mbox{d}x
    		&=\int_{\Omega}\eta f \mbox{d}x-\bar{f}\int_{\Omega}\eta \mbox{d}x
    		=\int_{\Omega}\eta f \mbox{d}x.
    	\end{align*}
    	Thus $\bar{f}$ disappears because $\eta$ has zero average; the test
    	function $f$ itself need not have zero average. H\"older's inequality and
    	\eqref{hessian}, with $1<a'\le2$, give
    	\begin{align*}
    		\left|\int_{\Omega}\eta f \mbox{d}x\right|
    		&\le\| G\|_{L^a}\|D^2\mathcal{N} f\|_{L^{a'}}\le\frac{C}{a'-1}\| G\|_{L^a} \| f\|_{L^{a'}}
    		=C(a-1)\| G\|_{L^a}\|f\|_{L^{a'}}.
    	\end{align*}
    	Taking the supremum over $\|f\|_{L^{a'}}\le1$ gives
    	$\|\eta\|_{L^a}\le C(a-1)\| G\|_{L^a}$.
    	To make this duality argument explicit, if $\eta\ne0$, the admissible test function
    	\[
    	f_\eta=\frac{|\eta|^{a-2}\eta}{\|\eta\|_{L^a}^{a-1}}
    	\]
    	satisfies $\|f_\eta\|_{L^{a'}}=1$ and $\int_{\Omega}\eta f_\eta \mbox{d}x=\|\eta\|_{L^a}$; the reverse inequality for all admissible test functions follows from H\"older's inequality.
    	
    	There is no need to assume the desired $L^a$ regularity in this argument.
    	For $G\in L^a$, define
    	$\Lambda_G(f)=\int_{\Omega} G:D^2\mathcal{N} f \mbox{d}x$ on $L^{a'}$.
    	The preceding bound and $L^p$ duality produce a unique representing
    	$\eta\in L^a$. Since $\mathcal{N}1=0$, testing with $1$ gives
    	$\int_{\Omega}\eta\mbox{d}x=0$. For regular $G$ with $Gn=0$, the integration by
    	parts calculation identifies this representative with \eqref{eta}.
    	Finally, with the Frobenius matrix norm,
    	$|\rho u\otimes u|=\rho|u|^2$. We have proved
    	\begin{equation}\label{etaLp}
    		\|\eta\|_{L^a}\le Ca\|\rho|u|^2\|_{L^a},\qquad a\ge2.
    	\end{equation}

    	Therefore, following the argument in \cite{VK1995}, we obtain the following estimates for the solutions of \eqref{xi} and \eqref{eta}:
    	\begin{align}\label{2-4}
    		&\|\nabla \xi \|_{L^{2m}} \leq m C\|u\|_{L^{2mk}} \|\rho \|_{L^{\frac{2mk}{k-1}}},\nonumber\\[2mm]
    		& \|\nabla \xi \|_{L^{2-r}} \leq C\|\sqrt{\rho} u\|_{L^2}\|\rho \|^\frac{1}{2}_{L^{\frac{2-r}{r}}}, \nonumber\\[2mm]
    		& \| \eta\|_{L^{2m}}\le  m C\|u\|_{L^{4mk}}^2 \|\rho \|_{L^{\frac{2mk}{k-1}}}.
    	\end{align}
    	where $C$ is a positive constant independent of $k>1$, $m \geq 1$, and $r$ with $1 \geq 2r \geq 0$.

    	Moreover, we obtain (\cite{VK1995}, Lemma 1-3)
    	\begin{align}\label{2-5}
    		&\|u\|_{L^{2m}} \leq m^{\frac{1}{2}} C \|\nabla u\|_{L^2}, \quad m>2,\nonumber\\[2mm]
    		& \|\xi \|_{L^{2m}} \leq m^{\frac{1}{2}} C \|\nabla \xi \|_{L^{\frac{2m}{m+1}}}, \quad m>2,
    	\end{align}
    	where $C$ is independent of $m$.
    	
    	Since $\|\nabla u\|^2_{L^2} \leq \int_{\Omega}( (2 \mu + \lambda)|\mbox{div}u|^2+ \mu |\omega|^2 )\mbox{d}x :=\psi(t),$ we infer that
    	\begin{align}\label{2-6}
    		&\| \xi \|_{L^{2m}} \leq m^{\frac{1}{2}} C \|\rho \|^{\frac{1}{2}}_{L^{m}},\nonumber\\[2mm]
    		& \|\nabla \xi \|_{L^{2m}} \leq m^{\frac{3}{2}} k^{\frac{1}{2}} C \|\rho \|_{L^{\frac{2mk}{k-1}}}(\psi(t))^{\frac{1}{2}},\nonumber\\[2mm]
    		& \|\eta \|_{L^{2m}} \leq  m^{2} k C \psi(t) \|\rho \|_{L^{\frac{2mk}{k-1}}},
    	\end{align}
    	for $m>2$ and $k>1$.
    	
    	Using the definitions of $\xi (x,t)$ and $\eta(x,t)$, we obtain
    	\begin{equation}\label{2-7}
    		\Delta( \xi_t + \eta - F -K + \frac{1}{|\Omega|}\int_{\Omega} F \mbox{d}x)=0,
    	\end{equation}
    	where we define
    	\begin{equation}\label{2-8}
    		\Delta K = \mbox{div}( \rho \nabla \Phi), \quad \int_{\Omega} K \mbox{d}x=0, \quad (\nabla K \cdot n)|_{\pl \Omega}=0.
    	\end{equation}

    	The boundary values of $\xi (x,t)$ and $\eta(x,t)$, together with the boundary conditions for $K$, we have
    	$$ \pl_n(\xi_t + \eta - F -K + \frac{1}{|\Omega|}\int_{\Omega} F \mbox{d}x)|_{\pl \Omega}= \mu \pl_{\tau}(A (u \cdot \tau))|_{\pl \Omega},$$
    	therefore by the uniqueness of the zero-mean solution of the Neumann
    	problem, imply
    	\begin{equation}\label{2-9}
    		\xi_t + \eta -F -K + \frac{1}{|\Omega|} \int_{\Omega} F \mbox{d}x=h.
    	\end{equation}
    	
    	For each fixed time for which $u\in H^1(\Om)$, define $h$ as the
    	zero-mean weak solution of
    	\begin{equation}\label{h}
    		\Delta h=0,\qquad
    		(\partial_nh)|_{\pl \Omega}=(\mu\partial_\tau\big(A(u\cdot\tau)\big))|_{\pl \Omega},
    		\qquad\int_\Om h \mbox{d}x=0.
    	\end{equation}
    	This definition requires both compatibility and a solvability argument.
    	Write $g_\partial=A(u\cdot\tau)|_{\partial\Om}$. The trace theorem
    	gives $g_\partial\in H^{1/2}(\partial\Om)$, so
    	$k=\mu\partial_\tau g_\partial\in H^{-1/2}(\partial\Om)$.
    	Tangential differentiation on the closed boundary gives
    	\[
    	\langle k,1\rangle_{\partial\Om}
    	=-\mu\langle g_\partial,\partial_\tau1\rangle_{\partial\Om}=0.
    	\]
    	For smooth data this is simply the integral of a derivative around a
    	closed curve. It is exactly the Neumann compatibility condition
    	$\int_{\partial\Om}\partial_nh\mbox{d}s=\int_\Om\Delta h \mbox{d}x=0$.
    	
    	We now give a direct construction that also yields the required estimate.
    	Choose once and for all a vector field
    	$\mathcal{B} \in W^{1,\infty}(\Om;\mathbb R^2)$ with boundary trace $A\tau$,
    	and put $g= \mathcal{B} \cdot u\in H^1(\Om)$. Such a fixed extension exists by
    	the smoothness of the boundary and $A$. Only its boundary values are
    	prescribed; it need not be a unit tangent field inside $\Om$.
    	For smooth $g$ the orientation convention gives
    	\[
    	\Div\nabla^\perp g=0,\qquad
    	n\cdot\nabla^\perp g
    	=n_1\partial_2g-n_2\partial_1g=-\partial_\tau g.
    	\]
    	Integration by parts therefore gives the weak problem
    	\begin{equation}\label{hweak}
    		\int_\Om\nabla h\cdot\nabla\varphi \mbox{d}x
    		=-\mu\int_\Om\nabla^\perp g\cdot\nabla\varphi \mbox{d}x,
    		\qquad \varphi\in H^1(\Om).
    	\end{equation}
    	The same identity holds for $g\in H^1$ by approximation and the
    	trace duality just described. The functional on its right is bounded:
    	\[
    	\left|\mu\int_\Om\nabla^\perp g\cdot\nabla\varphi \mbox{d}x\right|
    	\le\mu\norm{\nabla g}{L^2}\norm{\nabla\varphi}{L^2}.
    	\]
    	On the zero-mean Hilbert space $V$, the left-hand bilinear form is coercive. The Lax--Milgram theorem yields a unique $h\in V$ satisfying \eqref{hweak} for test functions in $V$; subtracting the mean of a test function extends the identity to all of $H^1$. For compactly supported test functions, the right-hand side vanishes since $\Div\nabla^\perp g=0$, proving $\Delta h=0$ in distributions. The identity for general tests then identifies the weak normal derivative with $\mu\partial_\tau g_\partial$. Thus it solves \eqref{h}. If two zero-mean solutions existed, their difference $d$ would satisfy $\int_\Om|\nabla d|^2 \mbox{d}x=0$, and hence $d=0$.
    	
    	The construction is independent of the interior extension. Two choices
    	of $g$ with the same trace differ by $v\in H^1_0(\Om)$, and
    	$\int_\Om\nabla^\perp v\cdot\nabla\varphi \mbox{d}x=0$ for every
    	$\varphi\in H^1(\Om)$, by approximation of $v$ with compactly
    	supported smooth functions. Their weak problems are identical.
    	
    	For the estimate, take $\varphi=h$ in \eqref{hweak}:
    	\begin{align*}
    		\norm{\nabla h}{L^2}^2
    		&=-\mu\int_\Om\nabla^\perp g\cdot\nabla h \mbox{d}x \le\mu\norm{\nabla g}{L^2}\norm{\nabla h}{L^2}.
    	\end{align*}
    	Dividing by $\norm{\nabla h}{L^2}>0$ (the zero case is immediate) and applying Poincar\'e's inequality for functions with zero mean gives
    	\[
    	\norm h{H^1}\le C_\Om\norm{\nabla h}{L^2}
    	\le C_\Om\mu\norm{\nabla g}{L^2}.
    	\]
    	The product rule, with $\mathcal{B}$ fixed, gives
    	\begin{align*}
    		\partial_i g&=\sum_j\big((\partial_i \mathcal{B}_j)u_j+\mathcal{B}_j\partial_i u_j\big),\\
    		\norm{\nabla g}{L^2}
    		&\le C\big(\norm{D\mathcal{B}}{L^\infty}\norm u{L^2}
    		+\norm \mathcal{B}{L^\infty}\norm{\nabla u}{L^2}\big)
    		\le C_{\Om,A}\norm u{H^1}.
    	\end{align*}
    	Combining this with Lemma \ref{lem-poincare} gives
    	\[
    	\norm h{L^a}\le C\sqrt a\,
    	\norm h{L^2}^{2/a}\norm h{H^1}^{1-2/a}
    	\le C\sqrt a\,\norm h{H^1}.
    	\]
    	Consequently,
    	\begin{equation}\label{hLp}
    		\norm h{L^a}\le C\sqrt a\| u\|_{H^1} ,\qquad a\ge2.
    	\end{equation}

    	Applying $\eqref{origin}^1$, we obtain
    	\begin{equation}\label{2-10}
    		\xi_t + (2\mu + \lambda(\rho))\frac{1}{\rho}(\rho_t + \nabla \rho \cdot u) +P + \eta - K + \frac{1}{|\Omega|}\int_{\Omega} F \mbox{d}x - h=0.
    	\end{equation}
    	
    	Define the function $\theta(\rho)$ on $(0,+\infty)$ by
    	\begin{equation}\label{theta}
    		\theta(\rho)=\int_{1}^{\rho}\frac{1}{s}( 2\mu + \lambda(\rho))\mbox{d}s=2\mu \ln \rho + \frac{1}{\beta}(\rho^{\beta}-1),
    	\end{equation}
    	where we used that $\lambda(\rho)=\rho^{\beta}$.
    	
    	We then obtain
    	\begin{equation}\label{2-11}
    		\frac{\pl }{\pl t} ( \xi + \theta(\rho) ) + u \cdot \nabla ( \xi + \theta(\rho)) + P + \eta -K - u \cdot \nabla \xi + \frac{1}{|\Omega|} \int_{\Omega} F \mbox{d}x -h =0.
    	\end{equation}
    	
    	Define $f(x,t)$ by
    	\begin{equation}\label{2-12}
    		f=\max\{ 0 , \xi(x,t) + \theta(\rho(x,t)) \}, \quad (x,t)\in \bar{\Omega_T}
    	\end{equation}
    	
    We multiply equation \eqref{2-11} by $\rho f^{2m-1}$, where $m \geq 4$ is a natural number, and integrate over $\Omega$. Using the continuity equation and the boundary conditions for $u$, we obtain
    	\begin{align}\label{2-13}
    		& \frac{1}{2m} \frac{d}{dt} \int_{\Omega} \rho f^{2m} \mbox{d}x + \int_{\Omega} P \rho f^{2m-1} \mbox{d}x + \int_{\Omega} \rho \eta f^{2m-1} \mbox{d}x - \int_{\Omega}\rho K f^{2m-1} \mbox{d}x\nonumber\\[2mm]
    		& - \int_{\Omega} \rho(u \cdot \nabla) \xi  f^{2m-1} + \frac{1}{|\Omega|}\int_{\Omega} F \mbox{d}x \int_{\Omega} \rho f^{2m-1}\mbox{d}x - \int_{\Omega} \rho h f^{2m-1} \mbox{d}x= 0.
    	\end{align}
    	
    	Define
    	\begin{equation}\label{2-14}
    		Z(t)= ( \int_{\Omega}\rho f^{2m} \mbox{d}x )^{\frac{1}{2m}}.
    	\end{equation}
    	
    	Using the preceding estimates for $\xi$ and $\eta$, the positivity $\rho >0$, and H$\ddot{o}$lder's inequality, we estimate the relevant terms in \eqref{2-13} as follows:
    	
    	\begin{align}\label{2-15}
    		&|\int_{\Omega}\rho \eta f^{2m-1} \mbox{d}x| \leq C \|\rho \|^{\frac{1}{2m}}_{L^{2m\beta+1}} Z^{2m-1}(t) \|\eta\|_{L^{2m+\frac{1}{\beta}}} \leq C m^2 \|\rho \|^{1+\frac{1}{2m}}_{L^{2m\beta+1}} \psi(t) Z^{2m-1}(t),\nonumber\\[2mm]
    		& | \int_{\Omega} \rho (u \cdot \nabla)\xi f^{2m-1} \mbox{d}x| \leq C m^2 \|\rho\|^{1+\frac{1}{2m}}_{L^{2m\beta+1}} \psi(t) Z^{2m-1}(t),\nonumber\\[2mm]
    		& |\frac{1}{|\Omega|}\int_{\Omega} F \mbox{d}x \int_{\Omega} \rho f^{2m-1}\mbox{d}x| \leq C Z^{2m-1}(t)( 1 + (\psi(t))^{\frac{1}{2}} + \|\rho\|^{\frac{\beta}{2}}_{L^{2m\beta+1}} (\psi(t))^{\frac{1}{2}} ).
    	\end{align}

    	Recalling that $\Delta K=\mbox{div}( \rho \nabla \Phi )$, we apply the Poincar\'{e} inequalities (Lemma \ref{lem-poincare} and Lemma \ref{lem-poincare2}), Young's inequality, and H$\ddot{o}$lder's inequality to obtain
    	$$\|K\|_{L^{2m+1/\beta}} \leq C m^{\frac{1}{2}} \|\nabla K\|_{L^2} \leq C m^{\frac{1}{2}} \|\rho \nabla \Phi\|_{L^2} \leq Cm^2\|\rho \|_{L^{2m\beta +1}},$$
    	where we have again used Lemma \ref{lem-poincare}, together with Lemma \ref{lem-div-curl}, the identity $\Delta \Phi =\rho -b(x)$, and the boundedness of $b(x)$.
    	
    	Therefore, we obtain
    	\begin{align}\label{2-16}
    		&|\int_{\Omega}\rho K f^{2m-1}| \leq C m^2 \|\rho \|^{1+\frac{1}{2m}}_{L^{2m\beta+1}} Z^{2m-1}(t).
    	\end{align}
    	
    	For the last term, we have
    	\begin{align}
    		& |\int_{\Omega}\rho h f^{2m-1}| \leq C \|\rho \|^{\frac{1}{2m}}_{L^{2m\beta+1}} Z^{2m-1}(t) \|h\|_{L^{2m+\frac{1}{\beta}}} \nonumber\\[2mm]
    		& \leq C m^{\frac{1}{2}}\|\rho \|^{\frac{1}{2m}}_{L^{2m\beta+1}} Z^{2m-1}(t) \|u\|_{H^1} \leq C m^{\frac{1}{2}} \|\rho \|^{\frac{1}{2m}}_{L^{2m\beta+1}} Z^{2m-1}(t) (\psi(t))^{\frac{1}{2}}.
    	\end{align}
    	
    	Combining the preceding inequalities with the nonnegativity of the second term in \eqref{2-13}, Young's inequality, and the lower bound for the mass, we derive
    	\begin{equation}\label{2-17}
    		Z(t) \leq C( 1+ \int_{0}^{t} m^2 (\psi(t)+1) \|\rho\|^{1+\frac{1}{2m}}_{L^{2m\beta+1}} \mbox{d}\tau+ \int_{0}^{t}(\psi(t))^{\frac{1}{2}}\|\rho\|^{\frac{\beta}{2}}_{L^{2m\beta+1}}\mbox{d}\tau ),
    	\end{equation}
    	for $m \geq 4$.
    	    	
    	Following the argument in \cite{VK1995}(equations (31)--(34)), we obtain
    	\begin{equation}\label{2-18}
    		\|\rho\|^{\beta}_{L^{2m\beta+1}} \leq C (Z(t) +m^{\frac{\beta}{2 \beta-1}}).
    	\end{equation}
    	
    	Observe that the positive constant $C$ in the preceding inequalities is independent of $m \geq 4$. Thus, \eqref{2-17} and \eqref{2-18} imply
    	\begin{align}\label{2-19}
    		\|\rho\|^{\beta}_{L^{2m\beta+1}} \leq C(m^{\frac{\beta}{2 \beta-1}} + m^2 \int_{0}^{t} (\psi(t)+1) \|\rho\|^{1+\frac{1}{2m}}_{L^{2m\beta+1}} \mbox{d}\tau+\int_{0}^{t}(\psi(t))^{\frac{1}{2}}\|\rho\|^{\frac{\beta}{2}}_{L^{2m\beta+1}}\mbox{d}\tau ),
    	\end{align}
    	where $C$ remains independent of $m \geq 4$.
    	
    	Using the fact established in Lemma \ref{lem1} that $\psi (t) \in L^1(0,T)$ and applying Gronwall's inequality, we obtain
    	\begin{equation}\label{2-20}
    		\|\rho\|^{\beta}_{L^{2m\beta+1}} \leq C(m^{\frac{\beta}{2 \beta-1}} + m^2 \int_{0}^{t} (\psi(t)+1) \|\rho\|^{1+\frac{1}{2m}}_{L^{2m\beta+1}} \mbox{d}\tau).
    	\end{equation}
    	
    	Setting $y(t)=m^{-\frac{2}{\beta-1}}\|\rho\|_{L^{2m\beta+1}}(t)$ for $t \in [0,T]$, we obtain
    	\begin{equation}\label{2-21}
    		y^{\beta}(t) m^{\frac{2\beta}{\beta-1}} \leq C(m^{\frac{2\beta}{2 \beta-1}} + m^2m^{(1+\frac{1}{m})(\frac{2}{\beta-1})} \int_{0}^{t} (\psi(t)+1) y^{1+\frac{1}{m}}(\tau) \mbox{d}\tau).
    	\end{equation}
    	
    	Hereafter, quantities of the form $m^{m^{-\alpha}}$, with $\alpha >0$, which are uniformly bounded for $m \geq 1$, are absorbed into the constants. Consequently,
    	\begin{equation}\label{2-22}
    		y^{\beta}(t) \leq C(1 + \int_{0}^{t} (\psi(t)+1) y^{1+\frac{1}{m}} (\tau)\mbox{d}\tau),
    	\end{equation}
    	Since $\beta \geq \frac{4}{3}$ and $\psi(t) \in L^1(0,T)$ by Lemma \ref{lem1}, we have
    	\begin{equation}\label{2-23}
    		y^{\beta}(t) \leq C(1 + \int_{0}^{t} (\psi(t)+1) y^{\beta} (\tau) \mbox{d}\tau),
    	\end{equation} 
    	for $m \geq 4$.
    	
    	Hence
    	$$y(t)\leq C, \quad t\in [0,T].$$
    	
    	This proves \eqref{lem2-rho}, namely,
    	\begin{equation}\label{2-24}
    		\|\rho \|_{L^p} \leq C(T) p^{\frac{2}{\beta-1}}.
    	\end{equation}
    	
    	Estimate \eqref{lem2-phi} follows from \eqref{lem2-rho}, the equation $\Delta \Phi= \rho -b(x)$, and the boundedness of $b$.
    	
    \end{proof}
    
    Before proving Lemma \ref{lem4}, we recall Lemma 3.4 of \cite{FLL2022}, which gives the following bound for $\int_{\Omega} |u|^{\nu} |\nabla u|^2 \mbox{d}x$:
    \begin{lem}\label{lem3}
    	There exist positive constants $\tilde{\nu}$ and $C$, both depending only on $\Omega$, such that, for any $\nu \in (0,\tilde{\nu})$,
    	\begin{equation}\label{lem3-1}
    		\int_{\Omega} |u|^{\nu}|\nabla u|^2 \mbox{d}x \leq C \int_{\Omega} |u|^{\nu} ( |\mbox{div}u|^2+ |\omega|^2 ) \mbox{d}x.
    	\end{equation}
    \end{lem}

    \begin{lem}\label{lem4}
    	There exists a constant $C $ depending on $\Omega$, $T$, $\mu$, $\beta$, $\gamma$, $\|b(x)\|_{L^{\infty}}$, $\|\rho_0\|_{L^{\infty}}$, and $\|u_0\|_{H^1}$ such that
    	\begin{equation}\label{lem4-1}
    		\sup_{0 \leq t \leq T} \int_{\Omega} \rho |u|^{2+\nu} \mbox{d}x \leq C,
    	\end{equation}
    	where
    	\begin{equation}\label{lem4-2}
    		\nu := R_T^{-\frac{\beta}{2}}\nu_0,
    	\end{equation}
    	for some suitably small generic constant $\nu_0 \in (0,1)$ depending only on $\mu $ and $\Omega$.
    \end{lem}
    
    \begin{proof}
    	First, combining Lemma \ref{lem-div-curl}, Lemma \ref{lem2}, and Poincar\'{e}'s inequality (Lemma \ref{lem-poincare}), we obtain
    	\begin{equation}\label{4-1}
    		\|u\|^2_{H^1} \leq C\|\nabla u\|^2_{L^2} \leq C(\|\mbox{div}u\|^2_{L^2}+ \|\omega\|^2_{L^2}),
    	\end{equation}
    	which, together with Lemma \ref{lem1}, implies
    	\begin{equation}\label{4-2}
    		\int_{0}^{T} ( \|u\|^2_{H^1} + A_1^2 ) \mbox{d}t \leq C,
    	\end{equation}
    	where all the terms in $A_1$ are bounded by Lemma \ref{lem1}, with
    	$$A_1^2(t):= 1 + \int_{\Omega}( \omega^2(t) + \frac{F^2(t)}{2\mu + \lambda(\rho)}) \mbox{d} x + \int_{\partial\Omega} A |u|^2 \mbox{d} s.$$
    	
    	Recall that
    	$$\rho u_t + \rho (u \cdot \nabla )u + \nabla P= \nabla ( (2\mu + \lambda(\rho)) \mbox{div}u )+ \mu \nabla^{\bot} \omega + \rho \nabla \Phi.$$
    	
    	Taking the inner product with $|u|^{\nu}u$, integrating over $\Omega$, and integrating by parts, we obtain
    	\begin{align}\label{4-3}
    		&\frac{1}{2+\nu } \frac{d}{dt} \int_{\Omega} \rho |u|^{2+\nu} \mbox{d}x + \int_{\Omega} |u|^{\nu} ( (2\mu + \lambda(\rho))|\mbox{div} u|^2+ \mu \omega^2 )\mbox{d}x - \int_{\Omega} \rho \nabla \Phi \cdot u |u|^{\nu} \mbox{d}x\nonumber\\[2mm]
    		& + \mu \int_{\partial\Omega} A |u|^{2+\nu} \mbox{d} s \leq \nu \int_{\Omega} (2\mu + \lambda(\rho)) |\mbox{div}u||u|^{\nu} |\nabla u|\mbox{d}x + \nu \mu \int_{\Omega}|u|^{\nu} |\nabla u|^2\mbox{d} x + C \int_{\Omega} \rho^{\gamma} | u|^{\nu} |\nabla u| \mbox{d}x \nonumber\\[2mm]
    		&\leq \frac{1}{2} \int_{\Omega} |u|^{\nu}(\mbox{div}u)^2 \mbox{d}x + \frac{\nu_0^2(\mu +1)}{2} \int_{\Omega}|u|^{\nu}|\nabla u|^2 \mbox{d}x + \int_{\Omega} |\nabla u|^2 \mbox{d}x + C \int_{\Omega} \rho |u|^(2+\nu) \mbox{d}x + C \int_{\Omega} \rho^{\frac{4+2\nu}{2-\nu}\gamma-\frac{2\nu}{2-\nu}} \mbox{d}x,
    	\end{align}
    	which, together with \eqref{4-2}, Lemma \ref{lem2}, Lemma \ref{lem3}, and
    	\begin{align}\label{4-4}
    		&-\int_{\Omega} \rho \nabla \Phi \cdot u |u|^{\nu} \mbox{d}x \geq - \int_{\Omega} \rho |u|^{2+\nu} \mbox{d}x - C\int_{\Omega} \rho^{\frac{\nu +2}{\nu +1}} \nabla \Phi^{\nu +2}\mbox{d}x\nonumber\\[2mm]
    		& \geq - \int_{\Omega} \rho |u|^{2+\nu} \mbox{d}x - C \| \rho\|^{2\nu +4}_{L^{2\nu +4}}\geq - \int_{\Omega} \rho |u|^{2+\nu} \mbox{d}x - C,
    	\end{align}
    	yields
    	\begin{equation}\label{4-5}
    		\frac{1}{2+\nu } \frac{d}{dt} \int_{\Omega} \rho |u|^{2+\nu} \mbox{d}x + \int_{\Omega} |u|^{\nu} ( (2\mu + \lambda(\rho))|\mbox{div} u|^2+ \mu \omega^2 )\mbox{d}x \leq C(1+ \int_{\Omega} \rho |u|^{2+\nu} \mbox{d}x ).
    	\end{equation}

    	Gronwall's inequality yields \eqref{lem4-1} and completes the proof of Lemma \ref{lem4}.

    \end{proof}

    We will frequently use the following standard estimate for the $L^p$-norm of $\nabla u$:
    
    \begin{lem}\label{lem5}
    	For $p>2$ and $\epsilon>0$, there exists some positive constant $C(p,\epsilon)$	depending only on $p$, $\epsilon$, and $\Omega$ such that
    	\begin{equation}\label{lem5-1}
    		\|\nabla u\|_{L^p} \leq C(p,\epsilon)R_T^{\frac{1}{2}-\frac{1}{p}+\epsilon}A_1 \left( \frac{A^2_2}{A_1^2} \right)^{\frac{1}{2}-\frac{1}{p}} + C(p,\epsilon)R_T^{\epsilon}A_1.
    	\end{equation}
    \end{lem}
    
    \begin{proof}
    	
    	First, we rewrite the momentum equations as
    	\begin{equation}\label{5-1}
    		\rho \dot{u} =\nabla F + \mu \nabla^{\bot}\omega + \rho \nabla \Phi,
    	\end{equation}
    	which, together with the boundary conditions \eqref{boundary1} and \eqref{boundary2}, shows that $F$ solves the Neumann problem:
    	\begin{equation}\label{5-2}
    		\left\{
    		\begin{array}{lll}
    			\Delta F = \mbox{div}( \rho (\dot{u} - \nabla \Phi)), \quad \mbox{in} \,\Omega\\[2mm]
    			\frac{\pl F}{\pl n} = \rho \dot{u} \cdot n - \mu n^{\bot} \cdot \nabla ( A u \cdot n^{\bot}) - \rho \nabla \Phi \cdot n =- \rho (u \cdot \nabla) n \cdot u - \mu \pl_{\tau}(Au_{\tau}), \quad \mbox{on} \, \pl \Omega,
    		\end{array}
    		\right.
    	\end{equation}
    	and that $\omega$ solves the related Dirichlet problem:
    	\begin{equation}\label{omega}
    		\left\{
    		\begin{array}{lll}
    			\mu \Delta \omega = \nabla^{\bot}\cdot (\rho (\dot{u} - \nabla \Phi) ), \quad \mbox{in} \,\Omega\\[2mm]
    			\omega= -A u \cdot n^{\bot}, \quad \mbox{on} \, \pl \Omega,
    		\end{array}
    		\right.
    	\end{equation}
    	The standard $L^p$ estimates for elliptic equations (\cite{GT1983}, [\cite{NS2004}, Lemma 4.27]) then imply that, for $k \geq 0$ and $p \in (1,\infty)$,
    	\begin{equation}\label{5-3}
    		\|\nabla F\|_{W^{k,p}} + \|\nabla \omega\|_{W^{k,p}} \leq C(p,k,\Omega)( \|\rho (\dot{u}-\nabla \Phi)\|_{W^{k,p}} + \|\nabla u\|_{W^{k,p}} ).
    	\end{equation}
    	
    	In particular, we have
    	\begin{equation}\label{5-4}
    		\|\nabla F\|_{L^2} + \|\nabla \omega\|_{L^2} \leq C( \|\rho (\dot{u}-\nabla \Phi)\|_{L^2} + \|\nabla u\|_{L^2} ) \leq CR_T^{\frac{1}{2}}A_2 + CA_1,
    	\end{equation}
    	which, together with the Poincar\'{e} inequality (Lemma \ref{lem-poincare2}) and Lemma \ref{lem2}, yields
    	\begin{equation}\label{5-5}
    		\| F\|_{H^1} + \| \omega\|_{H^1} \leq C(\|\nabla F\|_{L^2} + \|(\nabla \omega,\omega)\|_{L^2} )+ \frac{C}{|\Omega|}\int_{\Omega}F \mbox{d}x \leq CR_T^{\frac{1}{2}}A_2 + CA_1.
    	\end{equation}
    	
    	Applying Lemma \ref{lem-div-curl} and the estimate for $\rho$ in Lemma \ref{lem2}, we obtain
    	\begin{align}\label{5-6}
    		\|\nabla u\|_{L^p} &\leq C(p)( \|\mbox{div}u\|_{L^p} + \|\omega\|_{L^p} )\nonumber\\[2mm]
    		&\leq C(p)( \|\frac{F}{2 \mu + \lambda(\rho)}\|_{L^p} + \|\omega\|_{L^p} )+ C(p)\nonumber\\[2mm]
    		& \leq C(p)( \|\frac{F}{2 \mu + \lambda(\rho)}\|_{L^2}^{\frac{2}{p}-\epsilon}\|F\|_{H^1}^{-\frac{2}{p}+1-\epsilon} + \|\omega\|_{L^2}^{\frac{2}{p}}\|\omega\|_{H^1}^{1-\frac{2}{p}}  )+C(p)\nonumber\\[2mm]
    		& \leq  C(p)( A_1^{\frac{2}{p}-\epsilon}\|F\|_{L^2}^{\epsilon}\|F\|_{H^1}^{1-\frac{2}{p}+1} + A_1^{\frac{2}{p}}\|\omega\|_{H^1}^{1-\frac{2}{p}}  )+C(p)\nonumber\\[2mm]
    		& \leq C(p,\epsilon)R_T^{\frac{\beta\epsilon}{2}}A_1^{\frac{2}{p}}( \|F\|_{H^1}+ \|\omega\|_{H^1})^{1-\frac{2}{p}} +C(p),
    	\end{align}
    	where, in the third inequality, we have used Lemma \ref{lem-poincare} to estimate $\|\frac{F}{2 \mu + \lambda(\rho)}\|_{L^p}$ and $\|\omega\|_{L^p}$.
    	
    	Combining \eqref{5-5} and \eqref{5-6}, we complete the proof.

    \end{proof}

    We can now establish an upper bound for $(1+\|\nabla u\|_{L^2})$ in terms of $R_T$. This estimate plays an important role in obtaining the upper bound for the density.
    
    \begin{prop}\label{lem6}
         For any $\alpha\in (0,1)$, there is a constant $C(\alpha)$ depending only on $\alpha$, $\Omega$, $T$, $\mu$, $\beta$, $\gamma$, $\|b(x)\|_{L^{\infty}}$, $\|\rho_0\|_{L^{\infty}}$, and $\|u_0\|_{H^1}$ such that
         \begin{equation}\label{lem6-1}
         	\sup_{0 \leq t \leq T} \log A^2_1(t) + \int_{0}^{T} \frac{A_2^2(t)}{A_1^2(t)} \mbox{d}t \leq C(\alpha)R_T^{1+\alpha}.
         \end{equation}
    \end{prop}

    \begin{proof}
    	We modify the ideas of \cite{FLL2022,HL2016} to overcome the difficulties arising from the boundary. First, direct calculations show that
    	\begin{equation}\label{6-1}
    		\nabla^{\bot}\cdot \tilde{u} = \nabla^{\bot}\cdot \dot{u} - \nabla^{\bot }\cdot \nabla \Phi= \frac{D}{Dt}\omega - (\pl_1 u \cdot \nabla)u_2 + (\pl_2 u \cdot \nabla)u_1= \frac{D}{Dt}\omega + \omega \mbox{div}u,
    	\end{equation}
    	and that
    	\begin{align}\label{6-2}
    		&\mbox{div}\tilde{u}= \mbox{div}\dot{u}- \Delta\Phi= \frac{D}{Dt}\mbox{div}u + (\pl_1 u \cdot \nabla )u_1 + (\pl_2 u \cdot \nabla)u_2- \Delta \Phi\nonumber\\[2mm]
    		&= \frac{D}{Dt}( \frac{F}{2\mu + \lambda(\rho)} )+ \frac{D}{Dt}( \frac{P}{2\mu + \lambda(\rho)} )- 2 \nabla u_1 \cdot \nabla^{\bot }u_2 + |\mbox{div} u|^2- \Delta\Phi.
    	\end{align}
    	
    	Multiplying
    	\begin{equation}\label{6-3}
    		\rho \tilde{u}= \nabla F + \mu \nabla^{\bot}\omega
    	\end{equation}
    	by $2\tilde{u}$, integrating over $\Omega$, and using \eqref{6-1}, \eqref{6-2}, and the boundary condition \eqref{boundary1}, we obtain
    	\begin{align}\label{6-4}
    		&2 A^2_2= \int_{\Omega} 2 \rho |\tilde{u}|^2 \mbox{d}x= \int_{\Omega} 2\nabla F\cdot \tilde{u}+ 2\mu \nabla^{\bot} \omega \cdot \tilde{u} \mbox{d}x\nonumber\\[2mm]
    		&= \int_{\Omega} (-2F\mbox{div}\tilde{u} -2\mu \nabla^{\bot}\cdot \tilde{u} \omega )\mbox{d}x + \int_{\partial\Omega} (2 F \tilde{u} \cdot n + 2 \mu \omega ( \tilde{u} \cdot n^{\bot} )) \mbox{d} s\nonumber\\[2mm]
    		&= -\frac{d}{dt} \int_{\Omega} (\frac{|F|^2}{2\mu + \lambda{\rho}} + \mu |\omega|^2)\mbox{d}x - \int_{\Omega} (2F(u \cdot \nabla)(\frac{F}{2\mu +\lambda(\rho)}) +2F \frac{D}{Dt}(\frac{P}{2\mu + \lambda(\rho)})  \nonumber\\[2mm]
    		& -4F\nabla u_1 \cdot\nabla^{\bot}u_2 + 2F|\mbox{div}u|^2 - 2F\Delta\Phi +2\mu (u\cdot \nabla) \omega \omega + 2\mu |\omega|^2\mbox{div}u )\mbox{d}x + \int_{\partial\Omega} (2 F \tilde{u} \cdot n + 2 \mu \omega ( \tilde{u} \cdot n^{\bot} )) \mbox{d} s \nonumber\\[2mm]
    		&= -\frac{d}{dt} \int_{\Omega} (\frac{|F|^2}{2\mu + \lambda{\rho}} + \mu |\omega|^2)\mbox{d}x - \mu\int_{\Omega}|\omega|^2\mbox{div}u \mbox{d}x + 4\int_{\Omega} F \nabla u_1 \cdot \nabla^{\bot}u_2 \mbox{d}x- \int_{\Omega}2F|\mbox{div}u|^2\mbox{d}x\nonumber\\[2mm]
    		&  - \int_{\Omega} \frac{(\beta-1)\lambda(\rho)-2\mu}{(2\mu +\lambda)^2}F^2\mbox{div}u\mbox{d}x - 2\beta \int_{\Omega} \frac{\lambda(\rho)P}{(2\mu + \lambda(\rho))^2}F\mbox{div}u \mbox{d}x + 2\gamma\int_{\Omega} \frac{P}{2 \mu +\lambda(\rho)} F \mbox{div}u \mbox{d}x\nonumber\\[2mm]
    		&+ 2\int_{\Omega}F \Delta\Phi\mbox{d}x + \int_{\partial\Omega} (-2 F u \cdot \nabla n \cdot u + 2 \mu \omega ( \tilde{u} \cdot n^{\bot} )) \mbox{d} s =: - \frac{d}{dt}A_1^2+ \frac{d}{dt} \int_{\partial\Omega} A |u|^2 \mbox{d}s + \sum_{i=1}^{9}I_i. 
    	\end{align}
    	
    	Using Lemma \ref{lem-poincare}, \eqref{4-1}, and H$\ddot{o}$lder's inequality, we estimate $I_1$ as follows:
    	\begin{equation}\label{6-5}
    		I_1 \leq C\|\omega\|^2_{L^4}\|\mbox{div}u\|_{L^2}\leq C\|\omega\|_{L^2}\|\nabla \omega\|_{L^2}A_1\leq CA_1^2\|\nabla\omega\|_{L^2}.
    	\end{equation}

    	Similarly, choosing $p \geq \frac{8\beta}{\alpha}$, we obtain
    	\begin{align}\label{6-6}
    		|I_2|+|I_3|&\leq C\|F\|_{L^p}\|\nabla u\|_{L^2}^{\frac{2(p-3)}{p-2}}\|\nabla u\|_{L^p}^{\frac{2}{p-2}}\nonumber\\[2mm]
    		&\leq C\|F\|_{L^2}^{\frac{2}{p}}\|F\|_{H^2}^{1-\frac{2}{p}}A_1^{\frac{2(p-3)}{p-2}}\left( R_T^{\frac{\beta\epsilon}{2}}A_1^{\frac{2}{p}}( \|F\|_{H^1}+ \|\omega\|_{H^1})^{1-\frac{2}{p}} +1  \right)^{\frac{2}{p-2}}\nonumber\\[2mm]
    		& \leq C(\alpha) R_T^{\frac{\alpha\beta}{2}}( \|F\|_{H^1}+ \|\omega\|_{H^1} ) A_1^2 + C(\alpha)A_1^2,
    	\end{align}
    	where we have used \eqref{5-6} to estimate $\|\nabla u\|_{L^p}$ in the first line and the definitions of $R_T$ and $A_1$ to obtain $\|F\|_{L^2}\leq CR_T^{\frac{\beta}{2}}A_1$.
    	
    	The next four terms are estimated by the same arguments, using the bound $\|\rho\|_{L^p}\leq C$ from Lemma \ref{lem2} and the boundedness of $b(x)$:
    	\begin{align}\label{6-7}
    		\sum_{i=4}^{7} &\leq C \int_{\Omega}\frac{F^2|\mbox{div}u|}{2\mu +\lambda{\rho}}\mbox{d}x + C \int_{\Omega}\frac{P}{2\mu + \lambda(\rho)}|F||\mbox{div}u|\mbox{d}x + C\|F\|_{L^{3}}\|\Delta\Phi\|_{L^{\frac{3}{2}}} \nonumber\\[2mm]
    		&\leq CA_1\|\frac{F^2}{2\mu +\lambda{\rho}}\|_{L^2} + CA_1\|F\|_{L^{2+\frac{4\gamma}{\beta}}} + C\|F\|_{L^{3}}( 1+\|\rho\|_{L^{\frac{3}{2}}} )\nonumber\\[2mm]
    		& \leq CA_1\|\frac{F}{(2\mu +\lambda{\rho})^{\frac{1}{2}}}\|_{L^2}^{1-\alpha}\|F\|^{1+\alpha}_{L^{\frac{2(1+\alpha)}{\alpha}}} + C(\alpha)A_1\|F\|_{H^1} \nonumber\\[2mm]
    		& \leq C(\alpha)A_1^{2-\alpha}\|F\|^{\alpha}_{L^2}\|F\|_{H^1} + C(\alpha)A_1\|F\|_{H^1}\nonumber\\[2mm]
    		& \leq C(\alpha)R_T^{\frac{\alpha\beta}{2}}A_1^2\|F\|_{H^1},
    	\end{align}
    	where $0<\alpha<1$.
    	
    	Applying the trace theorem (Lemma \ref{theorem-trace}), Lemma \ref{lem-poincare2}, and Lemma \ref{lem-div-curl}, we obtain
    	\begin{equation}\label{6-8}
    		|I_8| \leq C \|F\|_{H^1}\|u\|^2_{H^1}\leq C\|F\|_{H^1}\|\nabla u\|_{L^2}^2\leq C\|F\|_{H^1}A_1^2.
    	\end{equation}
    	
    	For $I_9$, we have
    	\begin{align}\label{6-8-1}
    		I_9 &= 2 \mu \int_{\partial\Omega} \omega (\tilde{u} \cdot n^{\bot}) \mbox{d}s = - 2 \mu \int_{\pl \Omega} (A (u \cdot n^{\bot}) \frac{d}{dt} (u \cdot n^{\bot}) + A (u \cdot n^{\bot}) (u \cdot \nabla)u \cdot n^{\bot} + A (u \cdot n^{\bot}) \nabla \Phi \cdot n^{\bot} ) \mbox{d} s \nonumber\\[2mm]
    		& = - \mu \frac{d}{dt} \int_{\partial\Omega} A (u \cdot n^{\bot})^2 \mbox{d} s - 2 \mu \int_{\partial\Omega} (A (u \cdot n^{\bot})^2 (n^{\bot} \cdot \nabla) u \cdot n^{\bot} + A (u \cdot n^{\bot}) \nabla \Phi \cdot n^{\bot} ) \mbox{d} s \nonumber\\[2mm]
    		& = - \mu \frac{d}{dt} \int_{\partial\Omega} A |u|^2 \mbox{d} s - 2 \mu \int_{\partial\Omega} (\frac{1}{3}A (n^{\bot} \cdot \nabla)(u \cdot n^{\bot})^3 + A (u \cdot n^{\bot})^2 (n^{\bot} \cdot \nabla) n^{\bot} \cdot u + A (u \cdot n^{\bot}) \nabla \Phi \cdot n^{\bot} ) \mbox{d} s \nonumber\\[2mm]
    		& \leq - \mu \frac{d}{dt} \int_{\partial\Omega} A |u|^2 \mbox{d} s - 2 \mu \int_{\Omega} \frac{1}{3} \mbox{div}(A \nabla^{\bot} (u \cdot n^{\bot})^3) \mbox{d} x- C(\|\nabla u\|^3_{L^{2}}+\|\rho\|^2_{L^2}+1) \nonumber\\[2mm]
    		& \leq - \mu \frac{d}{dt} \int_{\partial\Omega} A |u|^2 \mbox{d} s + C \int_{\Omega} |\nabla u||u|^2 \mbox{d} x- C(\|\nabla u\|^3_{L^{2}}+\|\rho\|^2_{L^2}+1) \nonumber\\[2mm]
    		& \leq - \mu \frac{d}{dt} \int_{\partial\Omega} A |u|^2 \mbox{d} s + C A^3_1.
    	\end{align}

    	Substituting the estimates \eqref{6-5}-\eqref{6-8-1} into \eqref{6-4}, we obtain, for any $\alpha \in (0,1)$,
    	\begin{align}\label{6-9}
    		\frac{d}{dt}A^2_1(t)+ A_2^2(t) &\leq C(\alpha)R_T^{\frac{\alpha\beta}{2}}( \|F\|_{H^1}+\|\omega\|_{H^1} )A^2_1+C(\alpha)A_1^2\nonumber\\[2mm]
    		&\leq C(\alpha)R_T^{\frac{1+\alpha\beta}{2}}A_2A_1^2+C(\alpha)A_1^3\nonumber\\[2mm]
    		&\leq C(\alpha)R_T^{1+\alpha\beta}A_1^4 + \frac{1}{2} A_2^2,
    	\end{align}
    	where, in the second inequality, we have used \eqref{5-5} to estimate $\|F\|_{H^1}+\|\omega\|_{H^1}$. Combining this with \eqref{4-2}, which gives $\int_{0}^{T}A_1^2\mbox{d}t\leq C$, proves \eqref{lem6-1} and completes the proof of Lemma \ref{lem6}.
    	
    \end{proof}
    
    With these preliminary estimates in hand, we turn to the central task of estimating the effective viscous flux $F$ defined in \eqref{Fdefinition}. This can be done by solving the Neumann problem \eqref{5-2} using the classical Green's function. However, except for a few special domains, it is difficult to express Green's function explicitly for a general domain $\Omega$, which poses a major obstacle to further calculations. To address this difficulty, we begin with the simplest case, where $\Omega$ is the unit disc $\mathbb{D}$. The explicit Neumann Green function for the unit disc (see \cite{STT2016}) provides a model for our proof. Moreover, by the Riemann mapping theorem [\cite{SS2003}, Chapter 9], every simply connected domain is conformally equivalent to the unit disc. We can therefore pull back Green's function for the unit disc through a conformal mapping and, in this sense, reduce the general case to the unit disc. More precisely, let $\psi=(\psi_1,\psi_2):\bar{\Omega}\to \bar{\mathbb{D}}$ be the conformal mapping satisfying \eqref{psi_1} and Lemma \ref{lem-cauchy-riemann}. We define the pull-back Green's function $\tilde{N}$ on $\Omega$ by
    \begin{align}\label{tildeN}
    	\tilde{N}(x,y):=N(\psi(x),\psi(y))=-\frac{1}{2\pi}\left[ \log|\psi(x)-\psi(y)|+\log\left| |\psi(x)| \psi(y)- \frac{\psi(x)}{|\psi(x)|}\right| \right].
    \end{align}
    As shown in Lemma \ref{lem7} below, the outward normal derivative of $\tilde{N}$ on $\pl \Omega$ is no longer constant. Thus, $\tilde{N}$ is not the Green's function of $\Omega$ in the classical sense, but it is sufficient for the subsequent calculations.
    
    \begin{lem}\label{lem7}
    	Let $n$ be the outward unit normal at $y_0\in \pl \Omega$, and let $\tilde{N}=\tilde{N}(x,y)$ be defined by \eqref{tildeN}. Then
    	\begin{equation}\label{lem7-1}
    	     \frac{\pl \tilde{N}}{\pl n}(x,y_0) = - \frac{1}{2\pi}| \nabla \psi_1(y_0)|.
    	\end{equation}
    \end{lem}
    
    \begin{proof}
    	
    	We omit the proof and refer the reader to \cite{FLL2022}, Lemma 3.6, for details.
    	
    \end{proof}
    
    Now we use the pull-back Green's function $\tilde{N}$ defined by \eqref{tildeN} to obtain a pointwise representation of $F$ in $\Omega$ via Green's identity as follows:

    \begin{lem}\label{lem8}
    	Let $F \in C^1(\bar{\Omega}) \cap C^2(\Omega)$ solve the problem \eqref{5-2}. Then, for $x \in \Omega$,
    	\begin{align}\label{lem8-1}
    		F(x) &= \int_{\Omega} \nabla_y \tilde{N}(x,y) \cdot \rho \tilde{u}(y) \mbox{d}y - \int_{\partial\Omega} \frac{\pl \tilde{N}}{\pl n}(x,y)F(y) \mbox{d}S_y \nonumber\\[2mm]
    		& - \mu \int_{\partial\Omega} \tilde{N}(x,y) n^{\bot}\cdot \nabla (Au\cdot n^{\bot}) \mbox{d}S_y.
    	\end{align}
    \end{lem}
    
    \begin{proof}
    	
    	Write $\tilde{N}(x,y)=N_1(x,y)+N_2(x,y)$, where
    	\begin{equation}\label{8-1}
    		N_1(x,y)= -\frac{1}{2\pi} \log|\psi(x)-\psi(y)|,
    	\end{equation}
    	and
    	\begin{equation}\label{8-2}
    		N_2(x,y)= -\frac{1}{2\pi} \log\left| |\psi(x)| \psi(y)- \frac{\psi(x)}{|\psi(x)|}\right|.
    	\end{equation}
    	
    	By Lemma \ref{lem-cauchy-riemann}, conformal mappings preserve harmonicity, so $N_2(x,\cdot)$ is harmonic in $\Omega$. Since $N_2(x,\cdot)$ has no singularities in $\Omega$, we may apply Green's second identity to $F$ and $N_2(x,\cdot)$ in $\Omega$ to obtain
    	\begin{equation}\label{8-3}
    		- \int_{\Omega} N_2 \Delta F \mbox{d}y = \int_{\partial\Omega} ( \frac{\pl N_2}{\pl n}F -N_2 \frac{\pl F}{\pl n} ) \mbox{d}S_y,
    	\end{equation}
    	where we have integrated by parts twice.
    	
    	Applying Green's second identity to $F$ and $N_1(x,\cdot)$ in $\Omega \backslash \tilde{B_r}(x)$, where $\tilde{B_r}(x)=\psi^{-1}(B_r(\psi(x)))$ is the inverse image under $\psi$ of the ball centered at $\psi(x)$ with radius $r$, we obtain, for sufficiently small $r$,
    	\begin{align}\label{8-4}
    		&\int_{\Omega \backslash \tilde{B_r}(x)} ( \Delta N_1 F -N_1 \Delta F ) \mbox{d}y\nonumber\\[2mm]
    		& = \int_{\partial\Omega}( \frac{\pl N_1}{\pl n}F -N_1 \frac{\pl F}{\pl n} ) \mbox{d}S_y - \int_{\pl \tilde{B_r}(x)} (\frac{\pl N_1}{\pl n}F -N_1 \frac{\pl F}{\pl n} ) \mbox{d}S_y.
    	\end{align}
    	
    	Letting $r \to 0$ and using the fact that $\Delta N_1=0$ in $\Omega \backslash \tilde{B_r}(x)$, we have
    	\begin{align}\label{8-5}
    		&-\int_{\Omega} N_1 \Delta F \mbox{d}y = \int_{\partial\Omega}( \frac{\pl N_1}{\pl n}F -N_1 \frac{\pl F}{\pl n} ) \mbox{d}S_y -\lim_{r\to 0} \int_{\pl \tilde{B_r}(x)} (\frac{\pl N_1}{\pl n}F -N_1 \frac{\pl F}{\pl n} ) \mbox{d}S_y.
    	\end{align}
    	
    	For sufficiently small $r$,
    	\begin{align}\label{8-6}
    		& | \int_{\pl \tilde{B_r}(x)} N_1 \frac{\pl F}{\pl n} \mbox{d} S_y| \leq \int_{\pl \tilde{B_r}(x)}|N_1(x,y)|\mbox{d}S_y \cdot \sup_{\pl \tilde{B_r}(x)} |\nabla F| \nonumber\\[2mm]
    		& \leq \int_{\pl B_r( \psi(x))} | \log|\psi(x)-\psi(y)|| \frac{1}{\nabla \psi_1(y)} \mbox{d}S_{\psi(y)} \cdot \sup_{\pl \tilde{B_r}(x)} |\nabla F| \nonumber\\[2mm]
    		& \leq Cr |\log r|\cdot \sup_{\pl \tilde{B_r}(x)}|\nabla \psi_1|^{-1} \sup_{\pl \tilde{B_r}(x)}|\nabla F| \to 0 \quad \mbox{as} \quad r \to 0,
    	\end{align}
    	and
    	\begin{align}\label{8-7}
    		&\int_{\pl \tilde{B_r}(x)} F \frac{\pl N_1}{\pl n} \mbox{d}S_y = \int_{\pl \tilde{B_r}(x)} n_y \cdot \nabla_y N_1(x,y) F(y)\mbox{d}S_y \nonumber\\[2mm]
    		& = - \frac{1}{2\pi} \int_{\pl \tilde{B_r}(x)} \frac{\nabla \psi(y)[\psi(y)-\psi(x)]}{|\nabla \psi_1(y)||\psi(x)-\psi(y)|} \cdot \frac{\nabla \psi(y)[\psi(x)-\psi(y)]}{|\psi(x)-\psi(y)|^2} F(y) \mbox{d}S_y \nonumber\\[2mm]
    		& = -\frac{1}{2\pi r} \int_{\pl \tilde{B_r}(x)} |\nabla \psi_1(y)|F(y) \mbox{d}S_y= -\frac{1}{2\pi r}  \int_{\pl B_r( \psi(x))} F( \psi^{-1}(\tilde{y})) \mbox{d}S_{\tilde{y}} \to - F(x) \quad \mbox{as} \quad r \to 0,
    	\end{align}
    	where we have used \eqref{psi_1} and the argument in \cite{FLL2022}, Lemma 3.6. Adding \eqref{8-3} and \eqref{8-5} and using \eqref{8-6} and \eqref{8-7}, we obtain
    	\begin{align}\label{8-8}
    		-F(x) &= \int_{\Omega} \tilde{N}(x,y) \mbox{div}(\rho \tilde{u})(y) \mbox{d}y+ \int_{\partial\Omega} \frac{\pl \tilde{N}}{\pl n}(x,y)F(y) \mbox{d}S_y - \int_{\partial\Omega} \tilde{N}(x,y)(\rho \tilde{u} \cdot n(y)-\mu n^{\bot} \cdot \nabla (Au\cdot n^{\bot}) )\mbox{d}S_y\nonumber\\[2mm]
    		& = - \int_{\Omega} \nabla_y \tilde{N}(x,y) \cdot \rho \tilde{u}(y) \mbox{d}y + \int_{\partial\Omega} \frac{\pl \tilde{N}}{\pl n}(x,y)F(y) \mbox{d}S_y+ \mu \int_{\partial\Omega} \tilde{N}(x,y) n^{\bot}\cdot \nabla (Au\cdot n^{\bot}) \mbox{d}S_y,
    	\end{align}
    	which yields \eqref{lem8-1} and completes the proof of Lemma \ref{lem8}.

    \end{proof}

    We now estimate $F$ using its representation in \eqref{lem8-1}. For the first boundary term, Lemma \ref{lem7} and Lemma \ref{lem-cauchy-riemann} give
    \begin{equation}\label{F-boundaryestimate}
    	\max_{x \in \bar{\Omega}} | \int_{\partial\Omega} \frac{\pl \tilde{N}}{\pl n}(x,y)F(y) \mbox{d}S_y| \leq C \|F\|_{H^1},
    \end{equation}
    where we used the trace theorem (Lemma \ref{theorem-trace}) for $F$.
    
    Besides,
    \begin{align}
    	&| \int_{\partial\Omega} \tilde{N}(x,y) n^{\bot}\cdot \nabla (Au\cdot n^{\bot}) \mbox{d}S_y| = |\int_{\Omega} \mbox{div}(\nabla^{\bot}(Au\cdot n^{\bot})\tilde{N}(x,y))\mbox{d}y| \nonumber\\[2mm]
    	& \leq C\int_{\Omega} |\nabla \tilde{N}(x,y)|·(|\nabla u|+|u|)\mbox{d}x \leq C \|\nabla \tilde{N}\|_{L^{\frac{4}{3}}} \|\nabla u\|_{L^4} \leq CR_T^{\frac{1}{4}+\epsilon} A_1^{\frac{1}{2}}A_2^{\frac{1}{2}} + C R_T^{\epsilon}A_1.
    \end{align}
    
    It follows from \eqref{5-6}, \eqref{lem6-1}, and \eqref{F-boundaryestimate} that
    \begin{align}\label{F-boundaryestimate-2}
    	&\int_{0}^{T} \max_{x \in \bar{\Omega}} | \int_{\partial\Omega} \frac{\pl \tilde{N}}{\pl n}(x,y)F(y) \mbox{d}S_y| \mbox{d}t + \int_{0}^{T} \max_{x \in \bar{\Omega}} |\int_{\partial\Omega} \tilde{N}(x,y) n^{\bot}\cdot \nabla (Au\cdot n^{\bot}) \mbox{d}S_y| \mbox{d} t\nonumber\\[2mm]
    	& \leq \int_{0}^{T}( \|F\|_{H^1} + CR_T^{\frac{1}{4}+\epsilon} A_1^{\frac{1}{2}}A_2^{\frac{1}{2}} + C R_T^{\epsilon}A_1)\mbox{d}t \leq C \int_{0}^{T}( R_T^{\frac{1}{2}}A_2+ R_T^{\frac{1}{4}+\epsilon} A_1^{\frac{1}{2}}A_2^{\frac{1}{2}} + R_T^{\epsilon}A_1 )\mbox{d}t\nonumber\\[2mm]
    	& \leq C R_T^{\frac{1}{2}+\frac{\alpha}{2}} \left((\int_{0}^{T} \frac{A_2^2}{A_1^2}\mbox{d}t)^{\frac{1}{2}}(\int_{0}^{T} A_1^2\mbox{d}t)^{\frac{1}{2}} + 1 \right) \leq C(\alpha) R_T^{1+\alpha},
    \end{align}
    with $\alpha \in (0,1)$.

    Using the continuity equation and the boundary condition $(u\cdot n)|_{\pl \Omega}=0$, we rewrite the first term on the right-hand side of \eqref{lem8-1}, in the representation of $-F$, as follows:
    \begin{align}\label{J}
    	& - \int_{\Omega} \nabla_y \tilde{N}(x,y) \cdot \rho \tilde{u}(y) \mbox{d}y = - \int_{\Omega} \nabla_y \tilde{N}(x,y) \cdot \left( \frac{\pl(\rho u)}{\pl_t} + \mbox{div}( \rho u \otimes u) - \rho \nabla \Phi\right) \mbox{d}y\nonumber\\[2mm]
    	&= - \frac{\pl}{\pl t} \int_{\Omega} \pl_{y_j} \tilde{N}(x,y) \rho u_j(y)\mbox{d}y + \int_{\Omega} \pl_{y_i}\pl_{y_j}\tilde{N}(x,y)\rho u_i u_j(y)\mbox{d}y + \int_{\Omega} \pl_{y_j} \tilde{N}(x,y) \rho \pl_{y_j}\Phi \mbox{d}y\nonumber\\[2mm]
    	& = -(\pl_t + u \cdot \nabla)\int_{\Omega} \pl_{y_j} \tilde{N}(x,y) \rho u_j(y)\mbox{d}y + J_1+J_2,
    \end{align}
    where
    \begin{equation}\label{J1}
    	J_1:= \int_{\Omega}[ \pl_{x_i} \pl_{y_j}\tilde{N}(x,y) u_i(x) + \pl_{y_i} \pl_{y_j} \tilde{N}(x,y) u_i(y)  ] \rho u_j(y) \mbox{d}y,
    \end{equation}
    and
    \begin{equation}\label{J2}
    	J_2 := \int_{\Omega} \pl_{y_j} \tilde{N}(x,y) \rho \pl_{y_j}\Phi \mbox{d}y.
    \end{equation}
    
    Next, we establish the following pointwise estimate for $J_1$.
    
    \begin{lem}\label{lem9}
    	For $J_1$ as in \eqref{J1}, there exists a generic positive constant $C(\Omega)$ such that for any $x \in \Omega$ with $\psi(x) \neq 0$,
    	\begin{align}\label{lem9-1}
    		|J_1(x)| &\leq C(\Omega) \int_{\Omega} \frac{\rho |u|^2(y)}{|x-y|} \mbox{d}y + C(\Omega) \int_{\Omega} \frac{|u(x)-u(y)|}{|x-y|^2}\rho |u|(y) \mbox{d}y \nonumber\\[2mm]
    		&+ C(\Omega) \int_{\Omega} \frac{|u(x')-u(y)|}{|x'-y|^2}\rho |u|(y) \mbox{d}y,
    	\end{align}
    	where
    	\begin{equation}\label{lem9-2}
    		x':= \psi^{-1}\left( \frac{\psi(x)}{|\psi(x)|} \right) \in \pl \Omega.
    	\end{equation}
    \end{lem}
    
    \begin{proof}
    	
    	We omit the proof and refer the reader to Proposition 3.2 in \cite{FLL2022}, whose proof relies mainly on Lemma \ref{lem-cauchy-riemann} and Lemma \ref{lem-conformal}.
    	
    \end{proof}
    
    To obtain precise control of $F$, we focus on the crucial terms of commutator type. Although these terms are not commutators in the usual sense and the classical $L^p$ theory is difficult to apply directly, we can obtain suitable estimates by an alternative method.
    
    \begin{lem}\label{lem10}
    	For any $\epsilon>0$, we have
    	\begin{equation}\label{lem10-1}
    		\int_{0}^{T} \max_{x \in \bar{\Omega}} \int_{\Omega} \frac{|u(x)-u(y)|}{|x-y|^2}\rho |u|(y) \mbox{d}y \mbox{d}t \leq C(\epsilon)R_T^{1+\frac{\beta}{4}+3\epsilon}.
    	\end{equation}
    \end{lem}

    \begin{proof}
    	
    	First, let $2<p<4$ be chosen later. By the Sobolev embedding theorem (Theorem 5 of [\cite{E2010}, Chapter 5]), we have, for any $x,y \in \bar{\Omega}$,
    	\begin{equation}\label{10-1}
    		|u(x)-u(y)| \leq C(p) \|\nabla u\|_{L^p} |x-y|^{1-\frac{2}{p}},
    	\end{equation}
    	which implies
    	\begin{equation}\label{10-2}
    		\int_{\Omega} \frac{|u(x)-u(y)|}{|x-y|^2}\rho |u| (y)\mbox{d}y \leq C(p) \|\nabla u\|_{L^p} \int_{\Omega} |x-y|^{-(1+\frac{2}{p})} \rho |u|(y)\mbox{d}y.
    	\end{equation}
    	
    	Let $\delta>0$ and $\epsilon_0 \in (0,\frac{p-2}{8})$ be chosen later. On the one hand, H$\ddot{o}$lder's inequality and Lemma \ref{lem-poincare} give
    	\begin{align}\label{10-3}
    		& \int_{|x-y|<2 \delta} |x-y|^{-(1+\frac{2}{p})} \rho |u|(y)\mbox{d}y \leq C(p) R_T ( \int_{|x-y|<2 \delta} |x-y|^{-(1+\frac{2}{p})(1+\epsilon_0)} \mbox{d}y)^{\frac{1}{1+\epsilon_0}} \|u\|_{L^{\frac{1+\epsilon_0}{\epsilon_0}}}\nonumber\\[2mm]
    		& \leq C(p) R_T \delta^{1-\frac{2}{p}- \frac{2\epsilon_0}{1+\epsilon_0}} ( \frac{1+\epsilon_0}{\epsilon_0} )^{\frac{1}{2}} \|u\|_{H^1} \leq C(p) R_T \epsilon_0^{-\frac{1}{2}} A_1 \delta^{1-\frac{2}{p}- \frac{2\epsilon_0}{1+\epsilon_0}},
    	\end{align}
    	where we have used \eqref{4-1} to estimate $\|u\|_{H^1}$.
    	
    	On the other hand, for $\nu = R_T^{-\frac{\beta}{2}}\nu_0$ as defined in \eqref{lem4-2}, another application of H$\ddot{o}$lder's inequality, together with Lemma \ref{lem4}, gives
    	\begin{align}\label{10-4}
    		&\int_{|x-y|> \delta} |x-y|^{-(1+\frac{2}{p})} \rho |u|(y)\mbox{d}y\nonumber\\[2mm]
    		&\leq C(p) \left( \int_{|x-y|> \delta} |x-y|^{-(1+\frac{2}{p})(\frac{2+\nu}{1+\nu})} \mbox{d}y \right)^{\frac{1+\nu}{2+\nu}}\left( \int_{\Omega} \rho^{2+\nu}|u|^{2+\nu}\right)^{\frac{1}{2+\nu}} \leq C(p) R_T^{\frac{1+\nu}{2+\nu}}\delta^{-\frac{2}{p}+\frac{\nu}{2+\nu}},
    	\end{align}
    	since the definition of $R_T$ gives
    	$$\left( \int_{\Omega} \rho^{2+\nu} |u|^{2+\nu} \right)^{\frac{1}{2+\nu}} \leq R_T^{\frac{1+\nu}{2+\nu}} \left( \int_{\Omega} \rho |u|^{2+\nu} \right)^{\frac{1}{2+\nu}} \leq CR_T^{\frac{1+\nu}{2+\nu}}.$$
    	
    	Now, we choose $\delta>0$ such that
    	\begin{equation}\label{10-5}
    		\delta^{-\frac{2}{p}+\frac{\nu}{2+\nu}}= A_1^{\frac{2}{p}},
    	\end{equation}
    	which implies
    	\begin{equation}\label{10-6}
    		A_1\delta^{1-\frac{2}{p}- \frac{2\epsilon_0}{1+\epsilon_0}}= A_1^{\frac{2}{p}},
    	\end{equation}
    	provided we set
    	\begin{equation}\label{10-7}
    		\epsilon_0= \frac{(p-2)\nu}{8+(6-p)\nu} \in (0,\frac{p-2}{8}),
    	\end{equation}
    	where we used $\nu \in (0,1)$ and $2<p<4$.
    	
    	Then, it follows from \eqref{10-3}-\eqref{10-6} that
    	\begin{align}\label{10-8}
    		& \int_{\Omega} |x-y|^{-(1+\frac{2}{p})} \rho |u|(y)\mbox{d}y \leq C(p) R_T \epsilon_0^{-\frac{1}{2}}A_1^{\frac{2}{p}}+ C(p)R_T^{\frac{1+\nu}{2+\nu}}A_1^{\frac{2}{p}}\nonumber\\[2mm]
    		& \leq C(p)R_T^{1+\frac{\beta}{4}}A_1^{\frac{2}{p}},
    	\end{align}
    	where, in the last line, we have used $\epsilon_0^{-\frac{1}{2}}\leq C(p)\nu^{-\frac{1}{2}}\leq C(p)R_T^{\frac{\beta}{4}}$, which follows from the definition of $\epsilon_0$ and the identity $\nu=R_T^{-\frac{\beta}{2}}\nu_0$, with $\nu_0$ a small constant in $(0,1)$.
    	
    	Combining \eqref{10-8}, \eqref{10-2}, and Lemma \ref{lem5} shows that for any $\epsilon>0$,
    	\begin{align}\label{10-9}
    		& \int_{\Omega} \frac{|u(x)-u(y)|}{|x-y|^2}\rho |u| (y)\mbox{d}y \leq C(p) \|\nabla u\|_{L^p}R_T^{1+\frac{\beta}{4}}A_1^{\frac{2}{p}}\nonumber\\[2mm]
    		& \leq R_T^{1+\frac{\beta}{4}}A_1^{\frac{2}{p}}\left(C(p,\epsilon)R_T^{\frac{1}{2}-\frac{1}{p}+\epsilon}A_1 \left( \frac{A^2_2}{A_1^2} \right)^{\frac{1}{2}-\frac{1}{p}} + C(p,\epsilon)R_T^{\epsilon}A_1\right)\nonumber\\[2mm]
    		& \leq C(p,\epsilon) R_T^{\frac{3}{2}-\frac{1}{p}+\epsilon+\frac{\beta}{4}}A_1^{1+\frac{2}{p}} \left( \frac{A^2_2}{A_1^2} \right)^{\frac{1}{2}-\frac{1}{p}} + C(p,\epsilon)R_T^{1+\epsilon+\frac{\beta}{4}}A_1^{1+\frac{2}{p}}.
    	\end{align}
    	
    	Finally, integrating with respect to $t$ and applying H$\ddot{o}$lder's inequality, we obtain
    	\begin{align}\label{10-10}
    		& \int_{0}^{T} \max_{x \in \bar{\Omega}} \int_{\Omega} \frac{|u(x)-u(y)|}{|x-y|^2}\rho |u|(y) \mbox{d}y \mbox{d}t\nonumber\\[2mm]
    		& \leq C(p,\epsilon)R_T^{\frac{3}{2}-\frac{1}{p}+\epsilon+\frac{\beta}{4}} \left( \int_{0}^{T}A_1^2\mbox{d}t\right)^{\frac{1}{2}+\frac{1}{p}}\left( \int_{0}^{T} \frac{A_2^2}{A_1^2}\mbox{d}t\right)^{\frac{1}{2}-\frac{1}{p}} +C(p,\epsilon)R_T^{1+\epsilon+\frac{\beta}{4}} \left( \int_{0}^{T}A_1^2\mbox{d}t\right)^{\frac{1}{2}+\frac{1}{p}}\nonumber\\[2mm]
    		& \leq C(p,\epsilon) R_T^{2-\frac{2}{p}+2\epsilon+\frac{\beta}{4} },
    	\end{align}
    	where, in the last line, we have used \eqref{4-2} to estimate $\int_{0}^{T}A_1^2\mbox{d}t$ and Lemma \ref{lem6} to estimate $\int_{0}^{T} \frac{A_2^2}{A_1^2}\mbox{d}t$. Choosing $p=\frac{2}{1-\epsilon}$ yields \eqref{lem10-1} and completes the proof of Lemma \ref{lem10}.

    \end{proof}

    We can now establish the upper bound for the density, which is essential to the overall argument.
    
    \begin{lem}\label{lem11}
    	Assume that \eqref{betagamma} holds. Then there exists some positive constant $C$ depending only on $\Omega,T,\mu,\beta, \gamma$, $\|b(x)\|_{L^{\infty}}$, $\|\rho_0\|_{L^{\infty}}$, and $\|u_0\|_{H^1}$ such that
    	\begin{equation}\label{lem11-1}
    		\sup_{0 \leq t \leq T}( \|\rho \|_{L^{\infty}}+ \|u\|_{H^1}) + \int_{0}^{T} ( \|\omega\|_{H^1}^2 + \|F\|_{H^1}^2 + \|\sqrt{\rho}\dot{u}\|^2_{L^2} ) \mbox{d}t \leq C.
    	\end{equation}
    \end{lem}

    \begin{proof}
    	
    	For $\theta(\rho)$ defined in \eqref{theta}, equations \eqref{origin}, \eqref{Fdefinition}, and \eqref{omegadefinition} give
    	\begin{equation}\label{11-1}
    		\frac{D}{Dt} \theta(\rho) = - 2\mu \mbox{div}u - \rho^{\beta}\mbox{div}u = -(2\mu + \lambda{\rho})\mbox{div}u = -F-P \leq -F,
    	\end{equation} 
    	where we used that $P=\rho^{\gamma}\geq 0$.
    	
    	This, together with Lemma \ref{lem8}, \eqref{J}, \eqref{F-boundaryestimate}, and Lemma \ref{lem9}, gives
    	\begin{align}\label{11-2}
    		\frac{D}{Dt} \theta(\rho) & \leq -\frac{D}{Dt} \int_{\Omega} \pl_{y_j} \tilde{N}(x,y) \rho u_j(y) \mbox{d}y + \max_{x \in \bar{\Omega}} | \int_{\partial\Omega} \frac{\pl \tilde{N}}{\pl n} (x,y) F(y) \mbox{d}S_y| \nonumber\\[2mm]
    		&+ \max_{x \in \bar{\Omega}} |\int_{\partial\Omega} \tilde{N}(x,y) n^{\bot}\cdot \nabla (Au\cdot n^{\bot}) \mbox{d}S_y| + |J_1|+|J_2|\nonumber\\[2mm]
    		& \leq  -\frac{D}{Dt} \int_{\Omega} \pl_{y_j} \tilde{N}(x,y) \rho u_j(y) \mbox{d}y + \max_{x \in \bar{\Omega}} | \int_{\partial\Omega} \frac{\pl \tilde{N}}{\pl n} (x,y) F(y) \mbox{d}S_y|\nonumber\\[2mm]
    		&+ \max_{x \in \bar{\Omega}} |\int_{\partial\Omega} \tilde{N}(x,y) n^{\bot}\cdot \nabla (Au\cdot n^{\bot}) \mbox{d}S_y| +\max_{x \in \bar{\Omega}} \int_{\Omega} \frac{\rho |u|^2(y)}{|x-y|} \mbox{d}y \nonumber\\[2mm]
    		& + \max_{x \in \bar{\Omega}} \int_{\Omega} \frac{|u(x)-u(y)|}{|x-y|^2}\rho |u|(y) \mbox{d}y + \max_{x \in \bar{\Omega}} | \int_{\Omega} \pl_{y_j} \tilde{N}(x,y) \rho \pl_{y_j}\Phi \mbox{d}y|.
    	\end{align}
    	
    	On the one hand, a direct calculation gives, for $\nu = R_T^{-\frac{\beta}{2}}$ as in Lemma \ref{lem4},
    	\begin{align}\label{11-3}
    		& | \int_{\Omega} \pl_{y_j} \tilde{N}(x,y) \rho u_j(y) \mbox{d}y| \leq C \int_{\Omega} |x-y|^{-1} \rho (y) |u(y)| \mbox{d}y\nonumber\\[2mm]
    		& \leq C ( \int_{\Omega} |x-y|^{-\frac{2+\nu}{1+\nu}} \mbox{d}y)^{\frac{1+\nu}{2+\nu}}( \int_{\Omega} \rho^{2+\nu} |u|^{2+\nu} \mbox{d}y)^{\frac{1}{2+\nu}} \leq C \nu^{-\frac{1+\nu}{2+\nu}}R_T^{\frac{1+\nu}{2+\nu}}(\int_{\Omega} \rho |u|^{2+\nu} \mbox{d}y)^{\frac{1}{2+\nu}}\nonumber\\[2mm]
    		& \leq C R_T^{(1+\frac{\beta}{2})\frac{1+\nu}{2+\nu}} \leq C R_T^{\frac{2+\beta}{3}},
    	\end{align}
    	where the first inequality follows from the definition of $\tilde{N}$ and Lemma \ref{lem-cauchy-riemann}. We use H$\ddot{o}$lder's inequality in the second line, the arguments in \eqref{10-3} and \eqref{10-4} in the third line, and \eqref{lem4-1} in the last inequality.
    	
    	Similarly, one obtains
    	\begin{align}\label{11-4}
    		& | \int_{\Omega} \pl_{y_j} \tilde{N}(x,y) \rho \pl_{y_j}\Phi \mbox{d}y| \leq  C R_T^{\frac{2+\beta}{3}},
    	\end{align}
    	where we used Lemma \ref{lem2}.
    	
    	On the other hand, Sobolev's inequality leads to
    	\begin{equation}\label{11-5}
    		\int_{\Omega} \frac{\rho |u|^2(y)}{|x-y|} \mbox{d}y \leq R_T( \int_{\Omega} |x-y|^{-\frac{3}{2}} \mbox{d}y )^{\frac{2}{3}}( \int_{\Omega} |u|^6\mbox{d} y)^{\frac{1}{3}} \leq CR_T\|u\|^2_{H^1},
    	\end{equation}
    	which, together with \eqref{4-1} for $\|u\|_{H^1}$ and Lemma \ref{lem1}, gives
    	\begin{equation}\label{11-6}
    		\int_{0}^{T} \max_{x \in \bar{\Omega}}\int_{\Omega} \frac{\rho |u|^2(y)}{|x-y|} \mbox{d}y \mbox{d}t \leq C R_T. 
    	\end{equation}
    	
    	Finally, integrating \eqref{11-2} with respect to $t$ and using \eqref{F-boundaryestimate-2} for the boundary term, Lemma \ref{lem10}, \eqref{11-3}, \eqref{11-4}, and \eqref{11-6}, we obtain
    	\begin{equation}\label{11-7}
    		R_T^{\beta} \leq C(\epsilon) R_T^{\max\{1+\frac{\beta}{4}+3\epsilon,\frac{2+\beta}{3},1+\epsilon\}}.
    	\end{equation}
    	
    	Since $\beta> \frac{4}{3}$, this in particular implies by choosing $\epsilon>0$ sufficiently small
    	\begin{equation}\label{11-8}
    		\sup_{0 \leq t \leq T} \|\rho\|_{L^{\infty}} \leq C,
    	\end{equation}
    	which, together with \eqref{6-9}, \eqref{4-2}, \eqref{5-5}, and Gronwall's inequality, completes the proof of Lemma \ref{lem11}.

    \end{proof}

    \section{A Priori Estimates (II): Lower and Higher Order Estimates} \label{section4}
    
    Having established the upper bound for the density, we turn to the lower and higher order estimates. The arguments are similar to those in \cite{HL2016}, but the boundary terms require additional care. We therefore focus on these terms and draw on ideas from \cite{HL2016,CL2023}. Following \cite{CL2023}, we first establish an estimate of Poincar\'{e} type to handle the boundary conditions. Its proof relies on the slip condition $(u \cdot n)|_{\pl \Omega}=0$.

    \begin{lem}\label{lem12}
    	For $p \geq 1$, there exist positive constants $C_1(p,\Omega)$ and $C_2(\Omega)$ such that
    	\begin{equation}\label{lem12-1}
    		\|\dot{u}\|_{L^p} \leq C_1 ( \|\nabla \dot{u}\|_{L^2} + \|\nabla u\|_{L^2}^2 ),
    	\end{equation}
    	\begin{equation}\label{lem12-2}
    		\|\tilde{u}\|_{L^p} \leq C_1 ( \|\nabla \dot{u}\|_{L^2} + \|\nabla u\|_{L^2}^2 +1),
    	\end{equation}
    	and
    	\begin{equation}\label{lem12-3}
    		\|\nabla \dot{u}\|_{L^2} \leq C_2( \|\mbox{div}\dot{u}\|_{L^2} + \|\mbox{curl}\dot{u}\|_{L^2} + \|\nabla u\|_{L^4}^2 ),
    	\end{equation}
    	\begin{equation}\label{lem12-4}
    		\|\nabla \tilde{u}\|_{L^2} \leq C_2( \|\mbox{div}\dot{u}\|_{L^2} + \|\mbox{curl}\dot{u}\|_{L^2} + \|\nabla u\|_{L^4}^2 +1).
    	\end{equation}
    \end{lem}

    \begin{proof}
    	
    	First, by the definition $n^{\bot}=(n_2,-n_1)$ and the boundary condition $(u\cdot n)|_{\pl \Omega}=0$, we have
    	\begin{equation}\label{12-1}
    		u= ( u \cdot n^{\bot} )n^{\bot} \quad \mbox{on} \quad \pl \Omega,
    	\end{equation}
    	Thus, on the boundary $\pl \Omega$,
    	\begin{equation}\label{12-2}
    		\dot{u} \cdot n = (u\cdot \nabla )u \cdot n= - (u \cdot \nabla )n \cdot u = -(u \cdot n^{\bot})( u \cdot \nabla )n \cdot n^{\bot}= (u \cdot n^{\bot})(u \cdot \nabla)n^{\bot} \cdot n,
    	\end{equation}
    	which gives
    	\begin{equation}\label{12-3}
    		(\dot{u} - (u \cdot n^{\bot})(u \cdot \nabla)n^{\bot})\cdot n=0.
    	\end{equation}
    	
    	Applying \eqref{12-3} and Poincar\'{e}'s inequality (Lemma \ref{lem-poincare}), we obtain
    	\begin{equation}\label{12-4}
    		\|\dot{u} - (u \cdot n^{\bot})(u \cdot \nabla)n^{\bot}\|_{L^{\frac{3}{2}}} \leq C\|\nabla(\dot{u} - (u \cdot n^{\bot})(u \cdot \nabla)n^{\bot})\|_{L^{\frac{3}{2}}},
    	\end{equation}
    	and hence
    	\begin{equation}\label{12-5}
    		\|\dot{u} \|_{L^{\frac{3}{2}}} \leq C( \|\nabla \dot{u}\|_{L^{\frac{3}{2}}} + \|\nabla u\|_{L^2}^2 ),
    	\end{equation}
    	where we used
    	\begin{equation}\label{12-6}
    		\|(u \cdot n^{\bot})(u \cdot \nabla)n^{\bot}\|_{L^{\frac{3}{2}}} \leq C \|u\|_{L^3}^2\leq C(\|u\|_{L^2}^2 + \|\nabla u\|_{L^2}^2) \leq C \|\nabla u\|_{L^2}^2,
    	\end{equation}
    	by Lemma \ref{lem-poincare} and Lemma \ref{lem-poincare2}. Similarly, the same argument and H$\ddot{o}$lder's inequality give $\|\nabla((u \cdot n^{\bot})(u \cdot \nabla)n^{\bot})\|_{L^{\frac{3}{2}}} \leq C\|\nabla u\|_{L^2}\|u\|_{L^6}\leq C\|\nabla u\|_{L^2}^2$.
    	
    	Combining this with the Sobolev embedding theorem, we obtain, for $p \geq 1$,
    	\begin{align}\label{12-7}
    		\|\dot{u}\|_{L^p} &\leq C( \|\dot{u}\|_{L^2} + \|\nabla \dot{u}\|_{L^2} )\nonumber\\[2mm]
    		& \leq C ( \|\dot{u} \|_{L^{\frac{3}{2}}} + \|\nabla \dot{u}\|_{L^{\frac{3}{2}}}) + C\|\nabla \dot{u}\|_{L^2} \leq C( \|\nabla \dot{u} \|_{L^2} + \|\nabla u\|_{L^2}^2),
    	\end{align}
    	which completes the proof of \eqref{lem12-1}.
    	
    	Similarly, Lemma \ref{lem2} gives
    	\begin{align}\label{12-8}
    		\|\tilde{u}\|_{L^p} &\leq C( \|\dot{u}\|_{L^p} + \|\nabla \Phi\|_{L^p} )\nonumber\\[2mm]
    		& \leq C( \|\nabla \dot{u} \|_{L^2} + \|\nabla u\|_{L^2}^2+1),
    	\end{align}
    	which is \eqref{lem12-2}.
    	
    	Finally, \eqref{lem12-3} follows directly from \eqref{12-3} and Lemma \ref{lem-div-curl}. The proof is completed by combining \eqref{lem12-3} with
       	\begin{equation}\label{12-9}
    		\|\nabla \tilde{u}\|_{L^2} \leq C( \|\nabla \dot{u}\|_{L^2} + \|\Delta\Phi\|_{L^2} )
    	\end{equation}
    	where $\|\Delta \Phi\|_{L^2} \leq C(\|\rho\|_{L^2}+1)\leq C$ follows from the estimates for $\rho$ and $b(x)$.

    \end{proof}

    We now derive the lower order a priori estimates.
    
    \begin{lem}\label{lem13}
    	There is a positive constant $C$ depending only on  $\Omega,T,\mu,\beta, \gamma$, $\|b(x)\|_{L^{\infty}}$, $\|\rho_0\|_{L^{\infty}}$, and $\|u_0\|_{H^1}$ such that
    	\begin{equation}\label{lem13-1}
    		\sup_{0 \leq t \leq T} \sigma \int_{\Omega} \rho |\tilde{u}|^2 \mbox{d}x + \int_{0}^{T} \sigma \|\nabla \tilde{u}\|^2_{L^2} \mbox{d}t \leq C,
    	\end{equation}
    	with $\sigma(t):= \min \{ 1,t\}$. Moreover, for any $p\in [1,\infty)$, there is a positive constant $C(p)$ depending only on  $p,\Omega,T,\mu,\beta, \gamma, \|\rho_0\|_{L^{\infty}}$, and $\|u_0\|_{H^1}$ such that
    	\begin{equation}\label{lem13-2}
    		\sup_{0 \leq t \leq T} \sigma \|\nabla u\|_{L^p}^2 \leq C(p).
    	\end{equation}
    \end{lem}

    \begin{proof}
    	
    	First, we apply $\tilde{u}_j[\frac{\pl}{\pl t}+ \mbox{div}(u \cdot)]$ to $\eqref{origin}^j_2$, sum over $j$, and integrate over $\Omega$. Integrating by parts and using the boundary conditions \eqref{boundary1} and \eqref{boundary2}, we obtain
    	\begin{align}\label{13-1}
    		\frac{1}{2} \frac{d}{dt} \int_{\Omega} \rho |\tilde{u}|^2 \mbox{d}x & = \int_{\Omega} (\tilde{u} \cdot \nabla F_t + \tilde{u}_j \mbox{div}(\pl_jF u)) \mbox{d}x\nonumber\\[2mm]
    		& + \mu \int_{\Omega} ( \tilde{u} \cdot \nabla^{\bot}\omega_t + \tilde{u}_j \mbox{div}( u (\nabla^{\bot}\omega)_j ) ) \mbox{d}x=:\tilde{J}_1+\tilde{J}_2,
    	\end{align}
    	where we have used $\eqref{origin}^2$, namely $\pl_t \rho + \mbox{div}( \rho u)=0$.
    	
    	We estimate $\tilde{J}_1$ and $\tilde{J}_2$ separately. First, integrating by parts and using the boundary conditions $(u \cdot n)|_{\pl \Omega}=0$ and $(\nabla \Phi\cdot n)|_{\pl \Omega}=0$, we obtain
    	\begin{align}\label{13-2}
    		\tilde{J}_1 &= \int_{\Omega} (\tilde{u} \cdot \nabla F_t + \tilde{u}_j \mbox{div}(\pl_jF u)) \mbox{d}x = \int_{\Omega}( \tilde{u} \cdot \nabla F_t + \tilde{u}_j ( u \cdot \nabla \pl_j F) + \tilde{u} \cdot \nabla F \mbox{div}u )\mbox{d}x\nonumber\\[2mm]
    		& = \int_{\Omega}( \tilde{u} \cdot \nabla F_t + \tilde{u} \cdot \nabla ( u \cdot \nabla F) + \tilde{u} \cdot \nabla F \mbox{div}u - \tilde{u} \cdot \nabla u \cdot \nabla F )\mbox{d}x\nonumber\\[2mm]
    		& = \int_{\partial\Omega}(F_t + u \cdot \nabla F )(u \cdot \nabla u \cdot n) \mbox{d}s - \int_{\Omega} (F_t + u \cdot \nabla F )\mbox{div}\tilde{u} \mbox{d}x + \int_{\Omega}( \tilde{u} \cdot \nabla F \mbox{div}u - \tilde{u} \cdot \nabla u \cdot \nabla F )\mbox{d}x\nonumber\\[2mm]
    		& = \int_{\partial\Omega}(F_t + u \cdot \nabla F )(u \cdot \nabla u \cdot n) \mbox{d}s + \int_{\Omega}( \tilde{u} \cdot \nabla F \mbox{div}u - \tilde{u} \cdot \nabla u \cdot \nabla F )\mbox{d}x- \int_{\Omega}(2\mu + \lambda(\rho))|\mbox{div}\tilde{u}|^2\mbox{d}x\nonumber\\[2mm]
    		&+ \int_{\Omega}(2\mu + \lambda(\rho))( \pl_iu_j\pl_ju_i\mbox{div}\tilde{u} + \Delta\Phi\mbox{div}\tilde{u} )\mbox{d}x + \int_{\Omega}( \lambda'(\rho)\rho |\mbox{div}u|^2- \gamma P \mbox{div}u )\mbox{div}\tilde{u}\mbox{d}x\nonumber\\[2mm]
    		&\leq \int_{\partial\Omega}(F_t + u \cdot \nabla F )(u \cdot \nabla u \cdot n) \mbox{d}s - \mu \int_{\Omega}|\mbox{div}\tilde{u}|^2\mbox{d}x + C\int_{\Omega}|\tilde{u}||\nabla F ||\nabla u|\mbox{d}x + C\|\nabla u\|_{L^4}^4+C,
    	\end{align}
    	where we have used H$\ddot{o}$lder's inequality and Lemma \ref{lem2}, which gives $\|\Delta\Phi\|_{L^p}\leq C$ and $\|\rho\|_{L^p}\leq C$ for $p \geq 1$.
    	
    	In deriving \eqref{13-2}, we use the following identity, obtained by direct calculation:
    	\begin{align}\label{13-3}
    		F_t+ u \cdot \nabla F &= \lambda'(\rho) \rho_t \mbox{div}u + (2\mu + \lambda(\rho)) \mbox{div}u_t + u \cdot \nabla((2\mu +\lambda(\rho)) \mbox{div}u )+ P_t + u \cdot \nabla P\nonumber\\[2mm]
    		&= (2\mu +\lambda(\rho))( \mbox{div}u_t + u \cdot \nabla \mbox{div}u) - \lambda'(\rho) \rho|\mbox{div}u|^2+ \gamma P \mbox{div}u \nonumber\\[2mm]
    		& = (2\mu + \lambda(\rho))\mbox{div}\tilde{u} - (2\mu + \lambda(\rho))\pl_iu_j\pl_ju_i+ (2\mu + \lambda(\rho))\Delta \Phi- \lambda'(\rho) \rho|\mbox{div}u|^2+ \gamma P \mbox{div}u,
    	\end{align}
    	where
    	\begin{equation}\label{13-4}
    		\mbox{div}\tilde{u} = \mbox{div}u_t + \pl_iu_j\pl_ju_i+ u \cdot \nabla \mbox{div}u - \Delta\Phi.
    	\end{equation}
    	
    	For the first term on the right-hand side of \eqref{13-1}, we have
    	\begin{align}\label{13-5}
    		\int_{\partial\Omega}F_t(u \cdot \nabla u \cdot n) \mbox{d}s& =- \frac{d}{dt}\int_{\partial\Omega}F(u \cdot \nabla n \cdot u) \mbox{d}s + \int_{\partial\Omega} (F(u_t \cdot \nabla n \cdot u)+F(u \cdot \nabla n \cdot u_t)) \mbox{d}s\nonumber\\[2mm]
    		& = - \frac{d}{dt}\int_{\partial\Omega}F(u \cdot \nabla n \cdot u) \mbox{d}s + \int_{\partial\Omega} ( F(\tilde{u} \cdot\nabla n \cdot u + u \cdot \nabla n \cdot \tilde{u}) \nonumber\\[2mm]
    		&- F((u\cdot \nabla u)\cdot\nabla  n \cdot u + u \cdot \nabla n \cdot(u \cdot \nabla u)) + F(\nabla \Phi \cdot \nabla n \cdot u + u \cdot \nabla n \cdot \nabla \Phi) ) \mbox{d}s.
    	\end{align}
    	
    	For the second term in the second line, we apply the trace theorem, Poincar\'{e}'s inequality (Lemma \ref{lem-poincare2}), Lemma \ref{lem12} for $\|\tilde{u}\|_{H^1}$, \eqref{5-6} for $\|F\|_{H^1}$, and Lemma \ref{lem11} to obtain
    	\begin{align}\label{13-6}
    	     &\int_{\partial\Omega} ( F(\tilde{u} \cdot\nabla n \cdot u + u \cdot \nabla n \cdot \tilde{u}) \mbox{d}s \leq C \|u\|_{H^1}\|\tilde{u}\|_{H^1}\|F\|_{H^1}\nonumber\\[2mm]
    	     & \leq C \|\nabla u\|_{L^2}(\|\nabla \tilde{u}\|_{L^2} + \|\nabla u\|_{L^2}^2 )( \|\sqrt{\rho} \tilde{u}\|_{L^2} + \|\nabla u\|_{L^2} )\nonumber\\[2mm]
    	     & \leq \epsilon \|\nabla \tilde{u}\|^2_{L^2} + C(\epsilon)( \|\sqrt{\rho}\tilde{u}\|_{L^2}^2+1 ),
    	\end{align}
    	where we used the estimates of $\|\rho\|_{L^{\infty}}$ and $\|\nabla \Phi\|_{L^p}$ for $p\geq 1$.
    	
    	For the first term in the third line, we use \eqref{12-3} to obtain
    	\begin{align}\label{13-7}
    		&|\int_{\partial\Omega} F(u\cdot \nabla u)\cdot\nabla  n \cdot u  \mbox{d}s|= |\int_{\partial\Omega} F(u\cdot n^{\bot}) n^{\bot} \cdot\nabla u_i \pl_i n_j u_j \mbox{d}s|= | \int_{\Omega}\nabla^{\bot} \cdot ( F(u\cdot n^{\bot}) \nabla u_i \pl_i n_j u_j) \mbox{d}x|\nonumber\\[2mm]
    		& = | \int_{\Omega}\nabla u_i \cdot \nabla^{\bot}( F(u\cdot n^{\bot}) \pl_i n_j u_j) \mbox{d}x|\leq C \int_{\Omega} |\nabla u|( |F||u|^2 + F|u||\nabla u|+ |u|^2|\nabla F| )\nonumber\\[2mm]
    		& \leq C\|\nabla u\|_{L^4}( \|F\|_{L^4}\|u\|_{L^4}^2 + \|F\|_{L^4}\|u\|_{L^4}\|\nabla u\|_{L^4}+ \|\nabla F\|_{L^2}\|u\|_{L^8}^2 )\nonumber\\[2mm]
    		&\leq C (\|\nabla u\|_{L^4}^4 + \|F\|_{H^1}^2+1)\leq C( \|\nabla u\|_{L^4}^4+ \|\sqrt{\rho}\tilde{u}\|_{L^2}^2+1 ).
    	\end{align}

    	Similarly, for the last term we have
    	\begin{align}\label{13-8}
    		& | \int_{\partial\Omega} F(\nabla \Phi \cdot \nabla n \cdot u + u \cdot \nabla n \cdot \nabla \Phi)  \mbox{d}s| \leq C( \|\nabla u\|_{L^4}^4+ \|\sqrt{\rho}\tilde{u}\|_{L^2}^2+1 ),
    	\end{align}
    	where we used Lemma \ref{lem2} for $\|\nabla \Phi\|_{L^p}$ and $\|\Delta\Phi\|_{L^p}$ for $p\geq 1$.
    	
    	Then, we obtain
    	\begin{align}\label{13-9}
    		\int_{\partial\Omega}F_t(u \cdot \nabla u \cdot n) \mbox{d}s& \leq - \frac{d}{dt} \int_{\partial\Omega} F(u \cdot \nabla n \cdot u)\mbox{d}s + \epsilon\|\nabla \tilde{u}\|_{L^2}^2\nonumber\\[2mm]
    		&+ C(\epsilon)( \|\nabla u\|_{L^4}^4+ \|\sqrt{\rho}\tilde{u}\|_{L^2}^2+1 ).
    	\end{align}
    	
    	Returning to \eqref{13-2} for $\tilde{J}_1$, we estimate the remaining second term in the last line as follows:
    	\begin{align}\label{13-10}
    		&\int_{\partial\Omega} u \cdot \nabla F (u \cdot \nabla u \cdot n)\mbox{d}s = \int_{\partial\Omega} (u \cdot n^{\bot})\cdot \nabla F (\tilde{u} \cdot n)\mbox{d}s = \int_{\Omega} \nabla^{\bot}\cdot ((u \cdot n^{\bot}) \nabla F (\tilde{u} \cdot n))\mbox{d}x \nonumber\\[2mm]
    		&= \int_{\Omega} \nabla F \cdot \nabla^{\bot}((u \cdot n^{\bot}) (\tilde{u} \cdot n))\mbox{d}x \leq C \int_{\Omega}|\nabla F|(|\tilde{u}||\nabla u|+|\nabla \tilde{u}||u|)\mbox{d}x \nonumber\\[2mm]
    		&\leq C \|\nabla F\|_{L^2}( \|\tilde{u}\|_{L^4}\|\nabla u\|_{L^4}+ \|\nabla \tilde{u}\|_{L^2}\|\nabla u\|_{L^4} ) \leq \epsilon\|\nabla \tilde{u}\|_{L^2}^2 + C(\epsilon)(1+ \|\sqrt{\rho}\tilde{u}\|_{L^2}^2)(\|\nabla u\|_{L^4}^4+1).
    	\end{align}

    	Combining the preceding estimates, we obtain
    	\begin{align}\label{13-11}
    		\tilde{J}_1& \leq -\mu \int_{\Omega} | \mbox{div}\tilde{u}|^2\mbox{d}x - \frac{d}{dt} \int_{\partial\Omega} F(u \cdot \nabla n \cdot u)\mbox{d}s \nonumber\\[2mm]
    		& + C \epsilon\|\nabla \tilde{u}\|_{L^2}^2 + C(\epsilon)(1+ \|\sqrt{\rho}\tilde{u}\|_{L^2}^2)(\|\nabla u\|_{L^4}^4+1).
    	\end{align}

    	Next, we check that
    	\begin{align}\label{13-12}
    		\tilde{J}_2& = \mu \int_{\Omega} ( \tilde{u} \cdot \nabla^{\bot}\omega_t + \tilde{u}_j \mbox{div}( u (\nabla^{\bot}\omega)_j ) ) \mbox{d}x\nonumber\\[2mm]
    		& =  \mu \int_{\Omega} ( \tilde{u} \cdot \nabla^{\bot}\omega_t + \tilde{u} \cdot \nabla^{\bot}(\mbox{div}(u \omega))- \tilde{u}_j \pl_k( (\nabla^{\bot}u_k)_j \omega) ) \mbox{d}x\nonumber\\[2mm]
    		&= - \mu \int_{\Omega} |\mbox{curl} \tilde{u}|^2 \mbox{d}x + \mu \int_{\Omega} \mbox{curl}\tilde{u} \mbox{curl}(u \cdot \nabla u- \nabla \Phi) \mbox{d}x- \mu \int_{\Omega} \mbox{curl}\tilde{u} \mbox{div}(u \omega)\mbox{d}x \nonumber\\[2mm]
    		&- \mu \int_{\Omega} \tilde{u}_j \pl_k( (\nabla^{\bot}u_k)_j \omega)  \mbox{d}x + \mu \int_{\partial\Omega} (\tilde{u} \cdot n^{\bot}) ( \omega_t + \mbox{div}u \omega + u \cdot \nabla \omega ) \mbox{d} s.
    	\end{align}
    	
    	For the boundary term, we use the boundary condition \eqref{boundary1} to obtain
    	\begin{align}\label{13-12-1}
    		& \mu \int_{\partial\Omega} (\tilde{u} \cdot n^{\bot}) ( \omega_t + \mbox{div}u \omega + u \cdot \nabla \omega ) \mbox{d} s \nonumber\\[2mm]
    		& = \mu \int_{\partial\Omega} ((\tilde{u} \cdot n^{\bot}) \dot{\omega} + \frac{1}{2 \mu +\lambda } (\tilde{u} \cdot n^{\bot})( F+ P )(-A u \cdot n^{\bot})) \mbox{d} s \nonumber\\[2mm]
    		&= \mu \int_{\partial\Omega} (-A(\tilde{u} \cdot n^{\bot})^2 - (\tilde{u} \cdot n^{\bot})(A\nabla  \Phi \cdot n^{\bot} + A u \cdot( u \cdot \nabla) n^{\bot}+ (u \cdot \nabla) A (u \cdot n^{\bot})) \nonumber\\[2mm]
    		& + \frac{1}{2 \mu +\lambda } (\tilde{u} \cdot n^{\bot})( F+ P )(-A u \cdot n^{\bot})) \mbox{d} s \nonumber\\[2mm]
    		& \leq  - \mu \int_{\partial\Omega} A(\tilde{u} \cdot n^{\bot})^2 \mbox{d} s  + C \int_{\partial\Omega} ((|F|+1) |\tilde{u}| |u| + (|u|^2+|\nabla \Phi|) |\tilde{u}| ) \mbox{d}s \nonumber\\[2mm]
    		& \leq - \mu \int_{\partial\Omega} A(\tilde{u} \cdot n^{\bot})^2 \mbox{d} s  + C ((\|F\|_{H^1}+1) \| \tilde{u}\|_{H^1} \|u\|_{H^1} + (\|u\|_{H^1}^2+\|\nabla \Phi\|_{H^1}) \|\tilde{u}\|_{H^1} ) \nonumber\\[2mm]
    		& \leq - \mu \int_{\partial\Omega} A(\tilde{u} \cdot n^{\bot})^2 \mbox{d} s  + \epsilon\|\nabla \tilde{u}\|^2_{L^2} + C(\epsilon)\|\nabla u\|_{L^4}^4 + C(\epsilon)\|\sqrt{\rho}\tilde{u}\|^2_{L^2} + C(\epsilon).
    	\end{align}

    	A direct calculation, using integration by parts and H$\ddot{o}$lder's inequality, gives
    	\begin{align}\label{13-13}
    	    &- \mu \int_{\Omega} |\mbox{curl} \tilde{u}|^2 \mbox{d}x + \mu \int_{\Omega} \mbox{curl}\tilde{u} \mbox{curl}(u \cdot \nabla u- \nabla \Phi) \mbox{d}x\nonumber\\[2mm]
    	    &- \mu \int_{\Omega} \mbox{curl}\tilde{u} \mbox{div}(u \omega)\mbox{d}x - \mu \int_{\Omega} \tilde{u}_j \pl_k( (\nabla^{\bot}u_k)_j \omega)  \mbox{d}x \nonumber\\[2mm]
    	    &= - \mu \int_{\Omega} |\mbox{curl} \tilde{u}|^2 \mbox{d}x + \mu \int_{\Omega} \mbox{curl}\tilde{u} (\nabla^{\bot}u)^T : \nabla u \mbox{d}x\nonumber\\[2mm]
    	    &- \mu \int_{\Omega} \mbox{curl}\tilde{u} \mbox{div}u \omega\mbox{d}x + \mu \int_{\Omega} \pl_k \tilde{u}_j  (\nabla^{\bot}u_k)_j \omega)  \mbox{d}x - \mu \int_{\partial\Omega} \tilde{u}_j (\nabla^{\bot} u_k)_j \omega n_k \mbox{d}s\nonumber\\[2mm]
    	    & \leq - \mu \int_{\Omega} |\mbox{curl} \tilde{u}|^2 \mbox{d}x + \epsilon\|\nabla \tilde{u}\|^2_{L^2} + C(\epsilon)\|\nabla u\|_{L^4}^4,
    	\end{align}
    	where the last inequality uses the following estimate for the boundary term:
    	\begin{align}\label{13-13-1}
    		&- \mu \int_{\partial\Omega} \tilde{u}_j (\nabla^{\bot} u_k)_j \omega n_k \mbox{d}s = - \mu \int_{\partial\Omega} \omega ((\tilde{u} \cdot n )((\nabla^{\bot} u_k n_k) \cdot n) + (\tilde{u} \cdot \tau )((\nabla^{\bot} u_k n_k) \cdot \tau) )\mbox{d}s \nonumber\\[2mm]
    		& = - \mu \int_{\partial\Omega} (\omega \tilde{u}_n (u \cdot \pl_{\tau} n ) + \omega \tilde{u}_{\tau}(\mbox{div}u - \pl_{\tau}(u \cdot n^{\bot})))\mbox{d}s \nonumber\\[2mm]
    		& \leq - \mu \int_{\partial\Omega} (-A (u \cdot n^{\bot}) \tilde{u}_n (u \cdot \pl_{\tau} n ) + A (u \cdot n^{\bot}) \tilde{u}_{\tau} \pl_{\tau}(u \cdot n^{\bot}))\mbox{d}s + \epsilon\|\nabla \tilde{u}\|^2_{L^2} + C(\epsilon)\|\nabla u\|_{L^4}^4 \nonumber\\[2mm]
    		& \leq - \mu \int_{\partial\Omega} A (u \cdot n^{\bot}) \tilde{u}_{\tau} \pl_{\tau}(u \cdot n^{\bot})\mbox{d}s + \epsilon\|\nabla \tilde{u}\|^2_{L^2} + C(\epsilon)\|\nabla u\|_{L^4}^4 \nonumber\\[2mm]
    		& \leq \epsilon\|\nabla \tilde{u}\|^2_{L^2} + C(\epsilon)\|\nabla u\|_{L^4}^4,
    	\end{align}
    	Here we use the trace theorem (Proposition \ref{lem-boundary-key}), which gives, for any $u$ satisfying the boundary condition $(u \cdot n)|_{\pl \Omega}=0$,
    	\begin{align}
    		\int_{\partial\Omega} |u \cdot \nabla f| \mbox{d} s \leq C \|f\|_{H^1}\|u\|_{H^1}.
    	\end{align}
    	
    	The second equality in \eqref{13-13-1} follows from the identities
    	\begin{align}
    		& (\nabla^{\bot} u_k n_k) \cdot n = - \pl_{\tau} u \cdot n = -\pl_{\tau}(u \cdot n) + u \cdot \pl_{\tau}n= u \cdot \pl_{\tau}n,\nonumber\\[2mm]
    		& (\nabla^{\bot} u_k n_k) \cdot \tau = \pl_n u \cdot n = \mbox{div} u - \pl_{\tau} u \cdot \tau = \mbox{div} u - \pl_{\tau} (u \cdot n^{\bot}) + u \cdot \pl_{\tau} n^{\bot}\nonumber\\[2mm]
    		& = \mbox{div} u - \pl_{\tau} (u \cdot n^{\bot}) + \frac{1}{2}(u \cdot n^{\bot})\pl_{\tau} |n^{\bot}|^2 = \mbox{div} u - \pl_{\tau} (u \cdot n^{\bot}),
    	\end{align}
    	where we used $u = (u \cdot n^{\bot})n^{\bot}$ along the boundary and $|n^{\bot}|^2=1$.
    	
    	In the third line of \eqref{13-13-1}, we have also used \eqref{13-12-1}, since the term $\int_{\pl \Omega} \mu \omega \tilde{u}_{\tau}\mbox{div} u$ has already been estimated in \eqref{13-12-1}.

    	Substituting the estimate for $\tilde{J}_1$ and \eqref{13-13} for $\tilde{J}_2$ into \eqref{13-1} gives
    	\begin{align}\label{13-14}
    		& \frac{1}{2} \frac{d}{dt} \int_{\Omega} \rho |\tilde{u}|^2 \mbox{d}x + \mu \int_{\Omega}( |\mbox{div}\tilde{u}|^2+ | \mbox{curl}\tilde{u}|^2 ) \mbox{d}x \nonumber\\[2mm]
    		& \leq - \frac{d}{dt}\int_{\partial\Omega} F(u \cdot \nabla n \cdot u) \mbox{d}s + C\epsilon\|\nabla \tilde{u}\|_{L^2}^2 + C(\epsilon)( 1+\|\sqrt{\rho}\tilde{u}\|_{L^2}^2)^2,
    	\end{align}
    	where we have used
    	\begin{equation}\label{13-15}
    		\|\nabla u\|_{L^4}^4 \leq C\|\sqrt{\rho} \tilde{u}\|_{L^2}^2 +C,
    	\end{equation}
    	which follows from Lemma \ref{lem5} and Lemma \ref{lem11}. Combining this with Lemma \ref{lem12} and choosing $\epsilon$ sufficiently small, we obtain
    	\begin{equation}\label{13-16}
    		\frac{d}{dt} \int_{\Omega} \rho |\tilde{u}|^2 \mbox{d}x + C_2^{-1}\mu \int_{\Omega}( |\mbox{div}\tilde{u}|^2+ | \mbox{curl}\tilde{u}|^2 ) \mbox{d}x\leq - \frac{d}{dt}\int_{\partial\Omega} 2F(u \cdot \nabla n \cdot u) \mbox{d}s +C( 1+\|\sqrt{\rho}\tilde{u}\|_{L^2}^2)^2.
    	\end{equation}

    	Multiplying \eqref{13-16} by $\sigma$ and using \eqref{6-8} to estimate $\int_{\partial\Omega} F(u \cdot \nabla n \cdot u)\mbox{d}s$, together with Lemma \ref{lem11} and Gronwall's inequality, yields \eqref{lem13-1}.
    	
    	Finally, it follows from Lemma \ref{lem5}, Lemma \ref{lem11}, and \eqref{lem13-1} that for $p \geq 2$,
    	\begin{equation}\label{13-17}
    		\sigma \|\nabla u\|_{L^p}^2 \leq C(p)\sigma\|\sqrt{\rho}\tilde{u}\|_{L^2}^2+ C(p) \leq C(p),
    	\end{equation}
    	which completes the proof of \eqref{lem13-2}.

    \end{proof}

    \begin{lem}\label{lem14}
    	Assume that \eqref{betagamma} holds. Then for any $p>2$, there is a positive constant $C$ depending only on $\Omega,p,T,\mu,\beta, \gamma$, $\|b(x)\|_{L^{\infty}}$, $\|\rho_0\|_{L^{\infty}}$, and $\|u_0\|_{H^1}$ such that
    	\begin{align}\label{lem14-1}
    		& \int_{0}^{T}( \|\nabla F\|_{L^p} + \|\nabla \omega\|_{L^p}+ \|\rho \tilde{u}\|_{L^p})^{1+\frac{1}{p}} \mbox{d}t + \int_{0}^{T} ( \|F\|_{L^{\infty}}^{\frac{3}{2}}+ \|\omega\|_{L^{\infty}}^{\frac{3}{2}} \mbox{d}t )\nonumber\\[2mm]
    		&+ \int_{0}^{T}t( \|\nabla F\|_{L^p}^2 + \|\nabla \omega\|_{L^p}^2 + \|\tilde{u}\|^2_{H^1} ) \mbox{d}t \leq C,
    	\end{align}
    	and that
    	\begin{equation}\label{lem14-2}
    		\inf_{(x,t)\in \Omega \times (0,T)} \rho(x,t) \geq C^{-1} \inf_{x\in \Omega } \rho_0 (x).
    	\end{equation}
    \end{lem}

    \begin{proof}

    	First, using Lemma \ref{lem12}, the bound for $\|\rho\|_{L^{\infty}}$ from Lemma \ref{lem11}, and Lemma \ref{lem-poincare} to estimate $\|\rho \tilde{u}\|_{L^p}$, we obtain
    	\begin{align}\label{14-1}
    		\|\rho \tilde{u}\|_{L^p} & \leq C \|\rho \tilde{u}\|_{L^2}^{\frac{2(p-1)}{p^2-2}}\|\tilde{u}\|_{H^1}^{\frac{p(p-2)}{p^2-2}}\leq C \|\rho \tilde{u}\|_{L^2} + C \|\rho \tilde{u}\|_{L^2}^{\frac{2(p-1)}{p^2-2}}\|\nabla \tilde{u}\|_{L^2}^{\frac{p(p-2)}{p^2-2}},
    	\end{align}
    	which, together with Lemma \ref{lem13} and \eqref{13-15}, implies that
    	\begin{align}\label{14-2}
    		& \int_{0}^{T} ( \|\rho \tilde{u}\|_{L^p}^{1+\frac{1}{p}} + t \|\tilde{u}\|^2_{H^1} ) \mbox{d}t \nonumber\\[2mm]
    		& \leq C \int_{0}^{T}( \|\sqrt{\rho}\tilde{u}\|_{L^2}^2 + t\|\nabla \tilde{u}\|_{L^2}^2 + t^{-1+\frac{2}{p^3-p^2-2p+2}} )\mbox{d}t \leq C.
    	\end{align}

    	The Gagliardo-Nirenberg inequality, \eqref{5-4}, and Lemma \ref{lem11} then yield
    	\begin{align}\label{14-3}
    		\|F\|_{L^{\infty}} + \|\omega\|_{L^{\infty}} & \leq C(\|F\|_{L^2}+ \|\omega\|_{L^2} + \|F\|_{L^2}^{\frac{p-2}{2(p-1)}}\|\nabla F\|_{L^p}^{\frac{p}{2(p-1)}}+ \|\omega\|_{L^2}^{\frac{p-2}{2(p-1)}}\|\nabla \omega\|_{L^p}^{\frac{p}{2(p-1)}})\nonumber\\[2mm]
    		& \leq C(\|\rho \tilde{u}\|_{L^p}^{\frac{p}{2(p-1)}}+1),
    	\end{align}
    	which, together with \eqref{14-2}, gives \eqref{lem14-1}.
    	
    	Finally, \eqref{lem14-2} is a direct consequence of \eqref{lem14-1} and $\eqref{origin}^1$.

    \end{proof}

    Having established the lower order a priori estimates, we turn to the higher order estimates. Since no additional boundary terms need to be controlled, we follow \cite{HL2016} to derive the remaining a priori estimates. For completeness, we outline the proof.

    \begin{lem}\label{lem15}
    	Assume that \eqref{betagamma} holds. Then, for $q>2$ as in Theorem \ref{thm1}, there is a constant $\tilde{C}$ depending only on $\Omega,T,q,\mu,\beta, \gamma$, $\|b(x)\|_{W^{1,q}}$, $\|\rho_0\|_{W^{1,q}}$, and $\|u_0\|_{H^1}$ such that
    	\begin{align}\label{lem15-1}
    		& \sup_{0 \leq t \leq T} ( \|\rho \|_{W^{1,q}}+ t\|u\|_{H^2}^2 )+ \int_{0}^{T}( \|\nabla^2 u\|_{L^q}^{1+\frac{q}{1}} + t\|\nabla^2 u\|_{L^q}^2+ t\|u_t\|^2_{H^1}) \mbox{d}t \leq \tilde{C}.
    	\end{align}
    \end{lem}

    \begin{proof}
    	
    	First, set $\Psi=(\Psi_1, \Psi_2)$, where $\Psi_i := (2\mu +\lambda(\rho))\pl_i \rho(i=1,2)$. It follows from $\eqref{origin}^1$ that $\Psi_i$ satisfies
    	\begin{equation}\label{15-1}
    		\pl_t \Psi_i + ( u \cdot \nabla )\Psi_i + ( 2\mu +\lambda(\rho))\nabla \rho \cdot \pl_i u+ \rho \pl_i F+ \rho \pl_i P +\Psi_i \mbox{div}u=0.
    	\end{equation}
    	
    	For $q>2$, multiplying \eqref{15-1} by $|\Psi|^{q-2}\Psi_i$ and integrating over $\Omega$ gives
    	\begin{align}\label{15-2}
    		& \int_{\Omega} \frac{1}{q} \frac{\pl }{\pl t} |\Psi|^q + \frac{1}{q}(u \cdot \nabla)|\Psi|^q + (2\mu +\lambda(\rho))(\Psi \cdot \nabla)u \cdot \nabla \rho |\Psi|^{q-2}\nonumber\\[2mm]
    		& + \rho \Psi \cdot \nabla F|\Psi|^{q-2}+ \rho \Psi \cdot \nabla P|\Psi|^{q-2} + \mbox{div}u|\Psi|^q \mbox{d}x=0,
    	\end{align}
    	Integrating by parts and using the boundary condition \eqref{boundary1} to eliminate the boundary term, we obtain
    	\begin{equation}\label{15-3}
    		\frac{d}{dt} \|\Psi\|_{L^q} \leq \tilde{C}( 1+\|\nabla u\|_{L^{\infty}} )\|\Psi\|_{L^q} + \tilde{C}\|\nabla F\|_{L^q},
    	\end{equation}
    	where we used the estimates for $\rho$ (Lemma \ref{lem2}).
    	
    	Next, using the standard $L^p$ estimate for the elliptic system with boundary condition \eqref{boundary1}, together with \eqref{14-3} for $\|\mbox{div} u\|_{L^{\infty}}$ and \eqref{5-4} for $\|\nabla F\|_{L^q}$ and $\|\nabla \omega\|_{L^q}$, we obtain
    	\begin{align}\label{15-4}
    		\|\nabla^2 u\|_{L^q} & \leq \tilde{C}( \|\nabla \mbox{div}u\|_{L^q} + \|\nabla \omega\|_{L^q} )\nonumber\\[2mm]
    		& \leq \tilde{C}( \|\nabla(2\mu +\lambda(\rho))\mbox{div}u\|_{L^q} + \|\mbox{div}u\|_{L^{\infty}}\|\nabla \rho\|_{L^q} + \|\nabla \omega\|_{L^q} )\nonumber\\[2mm]
    		& \leq \tilde{C}( \|\mbox{div}u\|_{L^{\infty}} +1)\|\nabla \rho\|_{L^q} + \tilde{C}\|\nabla F\|_{L^q}+ \tilde{C}\|\nabla \omega\|_{L^q}\nonumber\\[2mm]
    		& \leq \tilde{C}(\|\rho \tilde{u}\|_{L^q}+1)\|\nabla \rho\|_{L^q} + \tilde{C}\|\rho \tilde{u}\|_{L^q}.
    	\end{align}
    	
    	Lemma \ref{lem-nablau} and the preceding estimate then give
    	\begin{align}\label{15-5}
    		\|\nabla u\|_{L^{\infty}} & \leq \tilde{C}( \|\mbox{div}u\|_{L^{\infty}} + \|\omega\|_{L^{\infty}} )\log (e+\|\nabla^2 u \|_{L^q}) + \tilde{C}\|\nabla u\|_{L^2} + \tilde{C}\nonumber\\[2mm]
    		& \leq \tilde{C}(1+\|\rho \tilde{u}\|_{L^q})\log (e+\|\nabla \rho\|_{L^q}+ \|\rho \tilde{u}\|_{L^q}) + \tilde{C}.
    	\end{align}
    	
    	Noting that
    	\begin{equation}\label{15-6}
    		2\mu \|\nabla \rho\|_{L^q} \leq \|\Psi\|_{L^q} \leq \tilde{C}\|\nabla \rho\|_{L^q},
    	\end{equation}
    	we substitute \eqref{15-5} into \eqref{15-3} and use \eqref{5-4} to estimate $\|\nabla F\|_{L^q}$, obtaining
    	\begin{equation}\label{15-7}
    		\frac{d}{dt} \log (e+\|\Psi\|_{L^q})= \frac{\frac{d}{dt}\|\Psi\|_{L^q}}{e+\|\Psi\|_{L^q}} \leq \tilde{C}(1+\|\rho \tilde{u}\|_{L^q})\log (e+\|\Psi\|_{L^q}) + \tilde{C}\|\rho \tilde{u}\|_{L^q}^{1+\frac{1}{q}},
    	\end{equation}
    	which, together with Gronwall's inequality, \eqref{5-4}, Lemma \ref{lem14}, and \eqref{15-6}, yields
    	\begin{equation}\label{15-8}
    		\sup_{0 \leq t \leq T}\|\nabla \rho\|_{L^q} \leq \tilde{C}.
    	\end{equation}
    	
    	Combining \eqref{15-4}, \eqref{15-8}, and Lemma \ref{lem14} gives
    	\begin{equation}\label{15-9}
    		\int_{0}^{T}( \|\nabla^2u \|_{L^q}^{1+\frac{1}{q}} + t \|\nabla^2 u\|_{L^q}^2 )\mbox{d}t \leq \tilde{C}.
    	\end{equation}
    	
    	Finally, using Lemma \ref{lem12} for $\|\tilde{u}\|_{H^1}$, \eqref{13-15} for $\|\nabla u\|_{L^4}$, Lemma \ref{lem2} for $\|\nabla \Phi\|_{H^1}\leq \|\rho\|_{L^2}+\|\nabla \Phi\|_{L^2}\leq C$, Lemma \ref{lem11} for $\sup_{0 \leq t \leq T}\|u\|_{H^1}$, Lemma \ref{lem13} for $t\|\nabla \dot{u}\|$, and \eqref{15-9}, we obtain
    	\begin{align}\label{15-10}
    		\int_{0}^{T} t \|u_t\|^2_{H^1} \mbox{d}t & \leq \tilde{C} \int_{0}^{T}t( \|\tilde{u}\|_{H^1}^2 + \|u \cdot \nabla u\|_{H^1}^2 + \|\nabla \Phi\|_{H^1}^2)\mbox{d}t \nonumber\\[2mm]
    		& \leq \tilde{C} \int_{0}^{T}t( 1+ \|\nabla \dot{u}\|_{L^2}^2 + \|\nabla u\|_{L^4}^4 + \|u\|^2_{H^1}\|\nabla u\|_{L^q}^2 )\mbox{d}t \leq \tilde{C}.
    	\end{align}
    	
    	Using Lemma \ref{lem11} for $\sup_{0 \leq t \leq T}\|\nabla u\|_{L^2}$, \eqref{5-4} for $\|\nabla F\|_{L^2} + \|\nabla \omega\|_{L^2}$, and \eqref{15-8}, we obtain
    	\begin{align}\label{15-11}
    		\|\nabla^2 u\|_{L^2} & \leq \tilde{C}( \|\nabla \mbox{div}u\|_{L^2} + \|\nabla \omega\|_{L^2} )\nonumber\\[2mm]
    		& \leq \tilde{C}( \|\nabla (2\mu +\lambda(\rho))\mbox{div}u\|_{L^2}+ \|\mbox{div}u\|_{L^{\frac{2q}{q-2}}}\|\nabla \rho\|_{L^q} + \|\nabla \omega\|_{L^2} )\nonumber\\[2mm]
    		& \leq \tilde{C}\|\nabla F\|_{L^2} + \tilde{C}\|\nabla \omega\|_{L^2} + \tilde{C} + \tilde{C}\|\nabla u\|_{L^2}^{\frac{q-2}{q}}\|\nabla^2 u\|_{L^2}^{\frac{2}{q}}\nonumber\\[2mm]
    		& \leq \frac{1}{2}\|\nabla^2u\|_{L^2} + \tilde{C} \|\rho \tilde{u}\|_{L^2} + \tilde{C},
    	\end{align}
    	which, together with Lemma \ref{lem13}, gives
    	\begin{equation}\label{15-12}
    		\sup_{0 \leq t \leq T}t \|\nabla^2 u\|^2_{L^2}\leq \tilde{C}.
    	\end{equation}
    	
    	Combining the preceding estimates with Lemma \ref{lem11} yields \eqref{lem15-1} and completes the proof of Lemma \ref{lem15}.

    \end{proof}

    \section{Proofs of Theorems \ref{thm1} and \ref{thm2}}\label{section5}

    With these estimates established, Theorem \ref{thm1} can be proved by adapting the arguments of [\cite{FLL2022}, Theorem 1.1] and [\cite{HL2016}, Theorem 1.1]. We first establish the global existence of a strong solution $(\rho,u,\Phi)$ under the assumptions that \eqref{betagamma} holds and $(\rho_0,u_0,\Phi_0)$ satisfies \eqref{initialvalue}. The proof follows that of [\cite{HL2016}, Proposition 5.1], with minor modifications.

    \begin{prop}\label{prop1}
    	Assume that \eqref{betagamma} holds and that $(\rho_0,u_0,\Phi_0)$ satisfies \eqref{initialvalue}. Then there exists a unique strong solution $(\rho,u ,\nabla \Phi)$ to \eqref{origin}-\eqref{boundary2} in $\Omega\times(0,\infty)$ satisfying \eqref{localexistence} for any $T \in (0,\infty)$. In addition, for any $q>2$, $(\rho,u,\nabla \Phi)$ satisfies \eqref{lem15-1} with a positive constant $C$ depending only on $\Omega,T,q,\mu,\beta, \gamma$, $\|b(x)\|_{W^{1,q}}$, $\|\rho_0\|_{W^{1,q}}$, and $\|u_0\|_{H^1}$.
    \end{prop}

    \begin{proof}
    	
    	The standard local existence result, Lemma \ref{lem-localexistence}, ensures that the problem \eqref{origin}-\eqref{boundary2} with initial data $(\rho_0,u_0,\Phi_0)$ has a unique local solution $(\rho,u,\nabla \Phi)$ defined up to a positive time $T_0$, which may depend on $\inf_{x\in \Omega }\rho_0(x)$, and satisfying \eqref{localexistence} and \eqref{localinf}. We set
    	
    	\begin{equation}\label{p1-1}
    		T^*= \sup\{ T| \sup_{0 \leq t \leq T} (\|(\rho,u)\|_{H^2}+\|\nabla \Phi\|_{H^1}) < \infty \}.
    	\end{equation}
    	
    	Clearly, $T^* \geq T_0$. If $T^*< \infty$, we claim that there exists a positive constant $\hat{C}$ which may depend on $b(x)$, $T^*$ and $\inf_{x\in \Omega } \rho_0(x)$ such that, for all $0<T<T^*$,
    	\begin{equation}\label{p1-2}
    		\sup_{0 \leq t \leq T}\|\rho\|_{H^2} \leq \hat{C}.
    	\end{equation}
    	
    	This together with \eqref{lem15-1} contradicts \eqref{p1-1}. Thus, we have
    	\begin{equation}\label{p1-3}
    		T^*=\infty.
    	\end{equation}
    	
    	It remains to prove \eqref{p1-2}. First, standard calculations and \eqref{lem15-1} imply that, for any $T \in (0,T^*)$,
    	\begin{equation}\label{p1-4}
    		\inf_{(x,t)\in \Omega \times (0,T)} \rho(x,t) \geq \inf_{x\in \Omega } \rho_0(x) \exp^{-\int_{0}^{T}\|\mbox{div}u\|_{L^{\infty}}\mbox{d}t} \geq \hat{C}^{-1},
    	\end{equation}
    	Here and below, $\hat{C}$ denotes a generic positive constant depending on $\inf_{x\in \Omega } \rho_0(x)$ and $T^*$ but independent of $T$. In view of \eqref{initialvalue}, we define
    	\begin{equation}\label{p1-5}
    		\sqrt{\rho}\tilde{u}(x,t=0) = \rho_0^{-\frac{1}{2}}( \mu \Delta u_0+ \nabla ((\mu + \lambda(\rho_0))\mbox{div}u_0)- \nabla P(\rho_0)).
    	\end{equation}
    	
    	Integrating \eqref{13-16} with respect to $t$ over $(0,T)$ and using \eqref{initialvalue}, Lemma \ref{lem11}, and \eqref{p1-5}, we obtain
    	\begin{equation}\label{p1-6}
    		\sup_{0 \leq t \leq T} \int_{\Omega} \rho |\tilde{u}|^2 \mbox{d}x + \int_{0}^{T} \|\nabla \dot{u}\|^2_{L^2}\mbox{d}t \leq \hat{C}.
    	\end{equation}

    	Combining this with \eqref{15-11} for $\|\nabla^2 u\|_{L^2}$ and \eqref{lem15-1}, and recalling that $\Delta F= \mbox{div}( \rho \tilde{u})$ and $\mu \Delta \omega= \nabla^{\bot} \cdot (\rho \tilde{u})$, we obtain
    	\begin{align}\label{p1-7}
    		& \sup_{0 \leq t \leq T}( \|\nabla^2 u\|_{L^2} + \|\nabla F\|_{L^2}+ \|\nabla \omega\|_{L^2}) + \int_{0}^{T}( \|\nabla^2 F\|_{L^2}^2 + \|\nabla^2 \omega\|_{L^2}^2 )\mbox{d}t \nonumber\\[2mm]
    		& \leq \hat{C} \sup_{0 \leq t \leq T} \|\rho \tilde{u}\|_{L^2} + \hat{C}\int_{0}^{T} \|\nabla(\rho \tilde{u})\|_{L^2}^2 \mbox{d}t \nonumber\\[2mm]
    		& \leq \hat{C} + \hat{C}\int_{0}^{T}( \|\nabla \rho\|_{L^q}^2\|\tilde{u}\|_{L^{\frac{2q}{q-2}}}^2 + \|\nabla \tilde{u}\|_{L^2}^2 )\mbox{d}t \leq \hat{C} + \hat{C}\int_{0}^{T} \|\tilde{u}\|_{H^1}^2\mbox{d}t \leq \hat{C},
    	\end{align}
    	where we have used \eqref{5-3} with $k=1$ and $p=2$. The additional terms involving $u$ and $\tilde{u}$ are controlled by \eqref{p1-6}, Lemma \ref{lem15}, and \eqref{15-11}.

    	Next, applying $\nabla$ to \eqref{15-1} gives
    	\begin{equation}
    		\pl_t \Psi_i + ( u \cdot \nabla )\Psi_i + ( 2\mu +\lambda(\rho))\nabla \rho \cdot \pl_i u+ \rho \pl_i F+ \rho \pl_i P +\Psi_i \mbox{div}u=0,
    	\end{equation}
    	multiplying the resulting identity by $\nabla \Psi_i$, integrating by parts, and using \eqref{p1-4} and \eqref{p1-7}, we obtain
    	\begin{align}\label{p1-8}
    		\frac{d}{dt}\|\nabla \Psi\|_{L^2} &\leq \hat{C}(1+\|\nabla u\|_{L^{\infty}})( 1+\|\nabla \Psi\|_{L^2}+\|\nabla \rho\|_{L^4}^2+ \|\nabla^2 \rho\|_{L^2} )\nonumber\\[2mm]
    		&+ \hat{C}\| |\nabla \rho||\nabla^2u|\|_{L^2} + \hat{C}\| |\nabla \rho||\nabla F|\|_{L^2} + \hat{C}\|\nabla^2 F\|_{L^2}.
    	\end{align}
    	
    	The identity $\Psi_i=(2\mu +\lambda(\rho))\pl_i\rho$, together with Lemma \ref{lem15}, yields
    	\begin{align}\label{p1-9}
    		\|\nabla^2 \rho\|_{L^2}+\|\nabla \rho\|_{L^4}^2 & \leq \hat{C}\|\nabla \Psi\|_{L^2}+ \hat{C}\|\nabla \rho\|_{L^4}^2 \leq \hat{C}\|\nabla \Psi\|_{L^2}+ \hat{C} \|\nabla \rho\|_{L^{l}}^{\frac{l}{2}}\|\nabla^2 \rho\|_{L^2}^{\frac{4-l}{2}} \nonumber\\[2mm]
    		& \leq \hat{C}\|\nabla\Psi\|_{L^2} + \frac{1}{2}\|\nabla^2 \rho\|_{L^2}+\hat{C},
    	\end{align}
    	where we have chosen $l=\min\{4,q\}$ and used Poincar\'{e}'s inequality (Lemma \ref{lem-poincare}).

    	H$\ddot{o}$lder's inequality, Lemma \ref{lem-poincare}, Lemma \ref{lem15}, and \eqref{p1-7} then give
    	\begin{align}\label{p1-10}
    		\| |\nabla \rho||\nabla^2u|\|_{L^2} + \| |\nabla \rho||\nabla F|\|_{L^2} &\leq \hat{C}\|\nabla \rho\|_{L^q}\|\nabla^2 u\|_{L^2}^{1-\frac{2}{q}}\|\nabla^3 u\|_{L^2}^{\frac{2}{q}} + \hat{C}\|\nabla \rho\|_{L^q}\|\nabla F\|_{L^2}^{1-\frac{2}{q}}\|\nabla^2 F\|_{L^2}^{\frac{2}{q}}\nonumber\\[2mm]
    		& \leq \hat{C}(\epsilon) \epsilon\|\nabla^3 u \|_{L^2} + \hat{C}\|\nabla^2 F\|_{L^2}.
    	\end{align}
    	
    	Moreover, the $L^2$ estimate for the elliptic system gives
    	\begin{align}\label{p1-11}
    		\|\nabla^3 u\|_{L^2} &\leq \hat{C}\|\nabla^2 \mbox{div} u\|_{L^2}+ \hat{C}\|\nabla^2 \omega\|_{L^2}\nonumber\\[2mm]
    		& \leq \hat{C}\|\nabla^2 F\|_{L^2} + \hat{C}\|\nabla^2 \rho\|_{L^2}+ \hat{C}\|\nabla \rho\|_{L^4}^2+ \hat{C}\||\nabla \rho||\nabla^2 u|\|_{L^2} \nonumber\\[2mm]
    		&+ \hat{C}( \|\nabla^2 \rho\|_{L^2} + \|\nabla \rho\|_{L^4}^2 )\|\nabla u\|_{L^{\infty}} + \hat{C}\|\nabla^2 \omega\|_{L^2}. 
    	\end{align}
    	
    	Substituting this and \eqref{p1-9} into \eqref{p1-10} shows
    	\begin{align}\label{p1-12}
    		& \|\nabla^2 \rho\|_{L^2}+\|\nabla \rho\|_{L^4}^2 \leq \hat{C}\|\nabla \Psi\|_{L^2}+ \hat{C}\|\nabla \rho\|_{L^4}^2\nonumber\\[2mm]
    		& \leq \hat{C}\|\nabla^2 F\|_{L^2} + \hat{C}\|\nabla^2 \omega\|_{L^2} + \hat{C}(1+\|\nabla \Psi\|_{L^2})(1+\|\nabla u\|_{L^{\infty}}),
    	\end{align}
    	which together with \eqref{p1-8} gives
    	\begin{equation}\label{p1-13}
    		\frac{d}{dt}\|\nabla \Psi\|_{L^2} \leq \hat{C}\|\nabla^2 F\|_{L^2} + \hat{C}\|\nabla^2 \omega\|_{L^2} + \hat{C}(1+\|\nabla \Psi\|_{L^2})(1+\|\nabla u\|_{L^{\infty}}).
    	\end{equation}
    	
    	This combined with \eqref{p1-7} and Gronwall's inequality yields
    	\begin{equation}\label{p1-14}
    		\sup_{0 \leq t \leq T} \|\nabla \Psi\|_{L^2} \leq \hat{C},
    	\end{equation}
    	which, together with \eqref{p1-9}, implies \eqref{p1-2}. To complete the proof of Proposition \ref{prop1}, we combine the preceding estimates with the relation between $\rho$ and $\Phi$, namely,
    	\begin{equation}\label{p1-15}
    		\Delta \Phi = \rho - b(x),
    	\end{equation}
    	where $b(x) \in H^2$, and apply standard Neumann elliptic regularity with the boundary condition $(\nabla \Phi \cdot n)|_{\pl \Omega}=0$ to obtain
    	\begin{equation}\label{p1-16}
    		\sup_{0 \leq t \leq T}\|\nabla \Phi\|_{H^3}\leq C\sup_{0 \leq t \leq T}\|\Delta \Phi\|_{H^2} \leq C\sup_{0 \leq t \leq T}(\|\rho\|_{H^2}+\|b(x)\|_{H^2}) \leq \hat{C}.
    	\end{equation}

    \end{proof}

    \subsection{Proof of Theorem \ref{thm1}}
    
    The proof is based on Proposition \ref{prop1} and a standard approximation argument. Since the Neumann problem for the Poisson equation requires the compatibility condition
    \[
    \int_{\Omega}\rho_0\,dx=\int_{\Omega}b(x)\,dx,
    \]
    the approximation has to preserve the total charge.
    
    Let $b^\delta\in C^\infty(\overline{\Omega})$ be such that
    \[
    b^\delta>0,
    \qquad
    b^\delta\to b
    \quad\mbox{in } W^{1,q}(\Omega)
    \quad\text{as }\delta\to0.
    \]
    Let $\mathcal S_\delta\rho_0$ be a non-negative smooth approximation of
    $\rho_0$ in $W^{1,q}(\Omega)$ and set
    \[
    \widetilde\rho_0^\delta
    :=\mathcal S_\delta\rho_0+\delta,
    \qquad
    \rho_0^\delta
    :=
    \frac{\displaystyle\int_\Omega b^\delta\,dx}
    {\displaystyle\int_\Omega\widetilde\rho_0^\delta\,dx}
    \widetilde\rho_0^\delta .
    \]
    Then
    \[
    \rho_0^\delta>0,
    \qquad
    \int_\Omega\rho_0^\delta\,dx
    =
    \int_\Omega b^\delta\,dx,
    \]
    and
    \[
    \rho_0^\delta\to\rho_0
    \quad\mbox{in }W^{1,q}(\Omega).
    \]
    We also choose
    \[
    u_0^\delta\in H^2(\Omega)
    \]
    satisfying the boundary condition \eqref{boundary1} such that
    \[
    u_0^\delta\to u_0
    \quad\mbox{in }H^1(\Omega),
    \]
    and put
    \[
    m_0^\delta=\rho_0^\delta u_0^\delta.
    \]
    Finally, let $\Phi_0^\delta$ be the zero-mean solution of
    \[
    \begin{cases}
    	\Delta\Phi_0^\delta=\rho_0^\delta-b^\delta,
    	&x\in\Omega,\\[1mm]
    	\partial_n\Phi_0^\delta=0,
    	&x\in\partial\Omega,\\[1mm]
    	\displaystyle\int_\Omega\Phi_0^\delta\,dx=0.
    \end{cases}
    \]
    By the Neumann elliptic estimates,
    \[
    \nabla\Phi_0^\delta\to\nabla\Phi_0
    \quad\mbox{in }W^{2,q}(\Omega).
    \]
    
    For every fixed $\delta>0$, Proposition \ref{prop1} yields a unique global
    strong solution
    \[
    (\rho^\delta,u^\delta,\Phi^\delta)
    \]
    corresponding to the smooth data
    $(\rho_0^\delta,u_0^\delta,\Phi_0^\delta)$ and the doping profile
    $b^\delta$. Notice that all the a priori estimates established in
    Sections 3 and 4 depend only on the norms appearing in Theorem \ref{thm1},
    and, in particular, are independent of
    \[
    \inf_{x\in\Omega}\rho_0^\delta(x).
    \]
    Moreover, the above approximations may be chosen such that
    \[
    \sup_{\delta>0}
    \left(
    \|\rho_0^\delta\|_{W^{1,q}}
    +\|u_0^\delta\|_{H^1}
    +\|b^\delta\|_{W^{1,q}}
    \right)<\infty.
    \]
    Consequently, the estimates obtained in Sections 3 and 4 are uniform
    with respect to $\delta$. By the standard compactness argument as in
    \cite{FLL2022,HL2016}, passing to a subsequence if necessary,
    we may let $\delta\to0$ and obtain a global strong solution
    $(\rho,u,\nabla \Phi)$ satisfying all the regularity properties stated in
    Theorem \ref{thm1}.
    
    It remains to discuss uniqueness. We follow the argument of
    Germain \cite{G2011}. The only additional term compared with
    the compressible Navier--Stokes system is the Poisson force. Let $(\rho_i,u_i,\Phi_i)$, $i=1,2$, be two solutions with the same initial data, choose the zero-mean representatives \(\Phi_i\), and set
    \[
    \delta\rho=\rho_1-\rho_2,\qquad
    \delta u=u_1-u_2,\qquad
    \delta \Phi= \Phi_1- \Phi_2.
    \]
    Then
    \[
    \begin{cases}
    	\Delta\delta\Phi=\delta\rho,
    	&x\in\Omega,\\
    	\partial_n\delta\Phi=0,
    	&x\in\partial\Omega,\\
    	\displaystyle\int_\Omega\delta\Phi\,dx=0,
    \end{cases}
    \]
    and hence the Neumann elliptic estimate gives
    \[
    \|\nabla\delta\Phi\|_{L^2}
    \le C\|\delta\rho\|_{L^2}.
    \]
    Furthermore,
    \[
    \begin{aligned}
    	\rho_1\nabla\Phi_1-\rho_2\nabla\Phi_2
    	&=
    	\delta\rho\,\nabla\Phi_1
    	+\rho_2\nabla\delta\Phi .
    \end{aligned}
    \]
    Therefore, using the uniform bounds for $\rho_i$ and $\nabla\Phi_i$, we obtain, for every $\varepsilon>0$,
    \[
    \begin{aligned}
    	\left|
    	\int_\Omega
    	(\rho_1\nabla\Phi_1-\rho_2\nabla\Phi_2)
    	\cdot\delta u\,dx
    	\right|
    	\le
    	\varepsilon\|\nabla\delta u\|_{L^2}^2
    	+
    	C_\varepsilon
    	\left(
    	\|\delta\rho\|_{L^2}^2
    	+\|\delta u\|_{L^2}^2
    	\right).
    \end{aligned}
    \]
    Thus the Poisson coupling gives only a lower-order contribution to the
    difference energy inequality and can be absorbed into the same Gronwall argument as in \cite{G2011}. This establishes uniqueness and completes the proof of Theorem \ref{thm1}.

    \subsection{Proof of Theorem \ref{thm2}}
    
   As above, we approximate the initial data and the doping profile by a family of smooth positive approximations that preserve the total charge:
    \[
    (\rho_0^\delta,u_0^\delta,b^\delta)
    \]
    In particular,
    \[
    \rho_0^\delta>0,
    \qquad
    \int_\Omega\rho_0^\delta\,dx
    =
    \int_\Omega b^\delta\,dx,
    \]
    and
    \[
    \rho_0^\delta\to\rho_0\quad\mbox{in }L^p(\Omega)
    \quad\text{for every }1\le p<\infty,
    \]
    while
    \[
    u_0^\delta\to u_0\quad\mbox{in }H^1(\Omega),
    \qquad
    b^\delta\to b\quad\mbox{in }W^{1,q}(\Omega).
    \]
    Let $(\rho^\delta,u^\delta,\nabla \Phi^\delta)$ be the corresponding global
    strong solutions furnished by Theorem \ref{thm1}. The lower-order estimates required for the compactness argument, which are independent of \(\|\rho_0^\delta\|_{W^{1,q}}\), are uniform in \(\delta\). Hence, using the standard compactness arguments in \cite{FLL2022,HL2016}, we may extract a subsequence such that, for every finite $T>0$ and $1\le p<\infty$,
    \[
    \rho^\delta\to\rho
    \quad\mbox{strongly in }C([0,T];L^p(\Omega)),
    \]
    together with the weak and weak-star convergences required to pass to the limit, as in \cite{FLL2022,HL2016}.
    
    It remains only to verify the convergence of the Poisson coupling.
    We normalize
    \[
    \int_\Omega\Phi^\delta(x,t)\,dx
    =
    \int_\Omega\Phi(x,t)\,dx=0.
    \]
    Since
    \[
    \Delta(\Phi^\delta-\Phi)
    =
    (\rho^\delta-\rho)-(b^\delta-b),
    \qquad
    \partial_n(\Phi^\delta-\Phi)=0,
    \]
    the Neumann elliptic estimate implies, for every $1<p<\infty$,
    \[
    \|\nabla(\Phi^\delta-\Phi)\|_{W^{1,p}}
    \le
    C_p
    \left(
    \|\rho^\delta-\rho\|_{L^p}
    +\|b^\delta-b\|_{L^p}
    \right)
    \longrightarrow0.
    \]
    In particular,
    \[
    \nabla\Phi^\delta\to\nabla\Phi
    \quad\mbox{strongly in }
    C([0,T];W^{1,p}(\Omega))
    \]
    for every finite $p$. Consequently,
    \[
    \begin{aligned}
    	\|\rho^\delta\nabla\Phi^\delta
    	-\rho\nabla\Phi\|_{L^1(\Omega)}
    	&\le
    	\|\rho^\delta-\rho\|_{L^2}
    	\|\nabla\Phi^\delta\|_{L^2}
    	+
    	\|\rho\|_{L^2}
    	\|\nabla(\Phi^\delta-\Phi)\|_{L^2}\\
    	&\longrightarrow0
    \end{aligned}
    \]
    uniformly on every bounded time interval. Thus the Poisson force
    passes to the limit in the sense of distributions. All the remaining
    nonlinear terms can be treated exactly as in
    \cite{FLL2022,HL2016}. We therefore obtain a global weak solution
    $(\rho,u,\nabla \Phi)$ satisfying all the properties stated in
    Theorem \ref{thm2}.

	\vspace{2mm}

	{\bf Acknowledgements}
	This research is supported by National Natural Science Foundation of
	China (No. 12221001, No. 12371225, No. 12571242).
	
	{\bf Data availability}  Data sharing is not applicable to this article as no data sets were generated or analysed during the current study.\\
	
	{\bf Declarations}\\

	{\bf Conflict of interest}
	The authors declare that they have no conflict of interest. The authors also declare that this manuscript has not been published previously and will not be submitted elsewhere while it is under consideration.

\end{document}